\documentclass[reqno,10pt]{amsart}

\usepackage{amsmath,amsfonts,amssymb,amsxtra,latexsym,amscd,enumerate,enumitem,amsthm,verbatim,here}
\usepackage{mathtools}
\usepackage{array}
\usepackage{booktabs}
\usepackage[table,dvipsnames]{xcolor}
\usepackage{cancel}
\usepackage[margin=1.2in]{geometry}
\numberwithin{equation}{section}

\usepackage{hyperref}
\definecolor{myrefred}{RGB}{180,0,0}
\definecolor{myciteblue}{RGB}{0,70,140}
\hypersetup{
colorlinks=true,
pdfstartview={FitV null},
linkcolor=myrefred,
citecolor=myciteblue,
urlcolor=black,
pdftitle={Asymptotic behavior of the one-dimensional Vlasov-Poisson system with a repulsive harmonic potential},
pdfauthor={Wenrui Huang}
}
\definecolor{labelkey}{rgb}{0.6,0,0}
\newcommand{\tK}{\widetilde{\K}}
\newcommand{\tga}{\widetilde{\ga}}
\newcommand{\pr}{\partial}

\newcommand{\of}[1]{\ensuremath{\left(#1\right)}}
\newcommand{\off}[1]{\ensuremath{\left[#1\right]}}

\newcommand{\R}{\mathbb{R}}

\newcommand{\K}{\mathcal{K}}
\newcommand{\la}{\langle}
\newcommand{\ra}{\rangle}

\newcommand{\dd}{d}
\newcommand{\sgn}{\operatorname{sgn}}

\def\ga{\gamma}
\def\H{\mathcal{H}}

\def\nab{\nabla}
\def\les{\lesssim}

\def\D{\mathcal{D}}
\newtheorem{theorem}{Theorem}
\newtheorem{lemma}{Lemma}[section]

\newtheorem{proposition}[lemma]{Proposition}

\theoremstyle{remark}
\newtheorem{remark}[lemma]{Remark}
\def\oo{\infty}
\makeatletter
\newcommand{\sectionnotoc}[1]{%
\begingroup
\let\@tocwrite\@gobbletwo
\section*{#1}%
\endgroup
}
\makeatother

\newcommand{\cO}{\mathcal O}

\begin{document}
\title[One-dimensional Vlasov--Poisson with harmonic forcing]{Asymptotic behavior of the one-dimensional Vlasov--Poisson system with a repulsive harmonic potential}
\author{Wenrui Huang}
\address{Brown University}
\email{wenrui\_huang@brown.edu}
\begin{abstract}
We study the one-dimensional Vlasov--Poisson system with a repulsive harmonic potential. For repulsive self-interactions, we prove modified scattering for initial data of arbitrary size, assuming only velocity weights and one derivative; no smallness, symmetry, compact support, or spatial moments are required. In the lens variables of \cite{HK26}, the self-consistent force carries a nonintegrable time weight. We remove it by a momentum correction that solves a forced Burgers equation; the repulsive sign makes this correction monotone, which gives bounds for large data. We also construct modified wave operators for both signs of the interaction. Together with \cite{VR24,BVR25,HK26}, this gives a scattering theory in every dimension.

\end{abstract}
\maketitle
\setcounter{tocdepth}{1}
\tableofcontents

\section{Introduction}
We study the long time behavior of the one-dimensional Vlasov--Poisson system with a repulsive harmonic potential 
\begin{equation}\label{eq:VPH}
    \pr_t\mu+v\pr_x\mu+(x+\lambda E[\mu])\pr_v\mu=0,
\end{equation}
where
\begin{equation*}
    E[\mu](t,x)=\frac{1}{2}\iint_{\R^2}\sgn(x-y)\mu^2(t,y,v)\dd y\dd v,\qquad \pr_xE[\mu](t,x)=\int_{\R}\mu^2(t,x,v)\dd v.
\end{equation*}
Equation \eqref{eq:VPH} can also be written in Hamiltonian form as
\begin{equation*}
    \pr_t\mu+\{\mu,\H\}=0,
\end{equation*}
where the Poisson bracket is defined by $\{f,g\}=\pr_xf\pr_vg-\pr_vf\pr_xg$, and the associated Hamiltonian is
\begin{equation*}
    \H=\frac{v^2}{2}-\frac{x^2}{2}-\lambda\phi(t,x),\qquad \pr_x\phi=E[\mu],\qquad \pr_x^2\phi=\int_{\R}\mu^2(t,x,v)\dd v.
\end{equation*}
This equation describes the evolution of a collisionless particle system subject to an external force and a self-consistent interaction. The parameter $\lambda$ determines the sign and strength of the self-consistent interaction: $\lambda>0$ corresponds to the repulsive plasma case, while $\lambda<0$ corresponds to the attractive gravitational case. The external harmonic force remains repulsive in both cases; for $\lambda<0$, self-attraction therefore competes with the outward harmonic force.

The total mass is conserved:
\begin{equation}\label{eq:mass}
M:=\iint_{\R^2}\mu_0^2(x,v)\dd x\dd v=\iint_{\R^2}\mu^2(t,x,v)\dd x\dd v.
\end{equation}
The definition of the electric field consequently yields the uniform bound
\begin{equation}
\|E(t)\|_{L^\infty}\le\frac{M}{2}.
\end{equation}
The potential $-x^2/2$ is the simplest external potential with unstable trapping: the flow $\dot x=v$, $\dot v=x$ has a hyperbolic fixed point at the origin, with stable and unstable manifolds $\{v=-x\}$ and $\{v=x\}$, and every other trajectory escapes at an exponential rate. Near any nondegenerate saddle of a Hamiltonian the dynamics looks like this to leading order. For example, near the photon sphere of the Schwarzschild spacetime the radial null geodesic flow at fixed angular momentum linearizes, up to normalization, to $H(x,v)=\frac12v^2-\frac12x^2$. See \cite{VR24,BVR25}.

\subsection{Previous work}

The global existence theory for the classical Vlasov--Poisson system in three dimensions is well established. Bardos and Degond \cite{BD85} proved global existence for small initial data, while Lions and Perthame \cite{LP91} and Pfaffelmoser \cite{P92} subsequently established global existence for broad classes of large initial data under suitable regularity and decay assumptions.

The scattering problem depends strongly on the spatial dimension. Chae and Ha \cite{CH06} established linear scattering for small solutions in dimensions $d\ge4$. In three dimensions, Choi and Ha \cite{CH11} showed that the long-range Coulomb interaction obstructs linear scattering in $L^1$. Choi and Kwon \cite{CK16} proved modified scattering for small data, and Ionescu, Pausader, Wang, and Widmayer \cite{IPWW22} identified an explicit logarithmic correction determined by the asymptotic electric field. Using the Hamiltonian structure and a pseudoconformal transformation, Flynn, Ouyang, Pausader, and Widmayer \cite{FOPW23} gave another proof of small-data modified scattering and constructed modified wave operators. 

Pankavich \cite{Pan22} established modified scattering for multispecies solutions under a decay assumption on the electric field, without imposing smallness directly on the initial data. Refined asymptotic expansions for small-data solutions were obtained by Bigorgne and Velozo Ruiz \cite{BiVe}, while Schlue and Taylor \cite{ScTa25} studied inverse modified scattering and asymptotic expansions for smooth, compactly supported scattering profiles of arbitrary size. More recently, Bigorgne \cite{Big26} proved asymptotic stability of large spherically symmetric solutions in the repulsive case under sufficiently small perturbations that need not preserve symmetry.

In the presence of a boundary, the  author, Pausader, and Suzuki \cite{HuPaSu} established global existence and modified scattering for small, localized data in a class of unbounded $C^3$ convex domains with perfectly conducting walls. Modified scattering has also been established for small perturbations of a repulsive point charge, first in the radial setting by Pausader and Widmayer \cite{PW21}, and subsequently without spherical symmetry by Pausader, Widmayer, and Yang \cite{PW24}, who described the coupled asymptotic dynamics of the point charge and the surrounding gas. Kepka and Widmayer \cite{KW25} established global existence and modified scattering for small, radially symmetric perturbations of an attractive point mass, with the initial particle distribution suitably localized on hyperbolic Kepler trajectories.

We also refer to \cite{HX26} for  the modified scattering result of Vlasov-Hartree system with a broad range of interaction potentials.

External potentials give rise to further asymptotic regimes, including phase mixing along trapped trajectories. Rioseco and Sarbach \cite{RS20} proved relaxation of macroscopic observables for collisionless gases in external gravitational central potentials under suitable nondegeneracy assumptions. Chaturvedi and Luk \cite{CL22} proved quantitative phase mixing for a one-dimensional linear transport equation with an anharmonic confining potential. They subsequently investigated the gravitational Vlasov--Poisson system under an attractive external Kepler potential \cite{CL24}, obtaining linear phase mixing without symmetry assumptions and nonlinear phase mixing for small spherically symmetric solutions on long but finite time intervals.

For the Vlasov--Poisson system with a repulsive harmonic potential, Velozo Ruiz and Velozo Ruiz \cite{VR24} proved global existence and sharp decay estimates for small-data solutions in dimensions $d\ge2$. Bigorgne, Velozo Ruiz, and Velozo Ruiz \cite{BVR25} subsequently established modified scattering for small-data solutions in two dimensions. The author and Kwon \cite{HK26} introduced a hyperbolic lens transform and constructed the corresponding wave operators and scattering maps for small data.

The present work establishes modified scattering for the one-dimensional repulsive problem with initial data of \textit{arbitrary size} in the stated class, without symmetry, compact support, or spatial moment assumptions. Together with the existing higher-dimensional results, this extends the scattering theory for the Vlasov--Poisson system with a repulsive harmonic potential to every spatial dimension.

\subsection{Main results}
The first results concern the long-time behavior of \eqref{eq:VPH}.
\begin{theorem}\label{thm:main}
Assume $\lambda>0$. Let $\mu_0\in C^1_{x,v}(\R^2)$  and satisfy
\begin{equation}\label{eq:data}
    C_0:=\|\la v\ra\mu_0\|_{L^2}
    +\|\la v\ra^3\mu_0\|_{L^\infty}
    +\|\la v\ra\nabla_{x,v}\mu_0\|_{L^\infty}<\infty.
\end{equation}
Then there exists a unique global classical solution $\mu$ of \eqref{eq:VPH} with initial data $\mu_0$. Moreover, there exist $\mu_\infty\in L^2_{x,v}\cap L^\infty_{x,v}$ and $E_\infty=E[\mu_\infty]\in L^\infty_x$ such that
\begin{equation}\label{eq:mscatter}
    \|\mu(t,\mathcal{X}(t,x-v,v),\mathcal{V}(t,x-v,v))-\mu_\infty(x,v)\|_{L^\infty_{x,v}}
    \les e^{-t},\qquad \text{as }t\to\infty,
\end{equation}
where
\begin{equation*}
    \begin{split}
        \mathcal{X}(t,x,v)&=x\cosh t+v\sinh t-\frac{\lambda}{2}(1-e^{-2t})E_\infty(x+v),\\
        \mathcal{V}(t,x,v)&=v\cosh t+x\sinh t+\frac{\lambda}{2}(1-e^{-2t})E_\infty(x+v).
    \end{split}
\end{equation*}
\end{theorem}

\begin{remark}\leavevmode
\begin{enumerate}[label=(\arabic*),leftmargin=*,topsep=0.3em,itemsep=0.3em]
\item The global well-posedness conclusion of Theorem~\ref{thm:main} holds for both signs of $\lambda$, without any smallness assumption on the initial data. Indeed, the estimates for the Vlasov equation in Proposition~\ref{prop:time} depend on $|\lambda|$ and give a unique classical solution in the lens variables on $[0,1)$. Since $s=\tanh t$, this yields a unique global classical solution of \eqref{eq:VPH} for $t\ge0$ under the same initial-data assumptions, including in the attractive case $\lambda<0$. 
  \item The repulsive sign $\lambda>0$ is crucial to our large-data argument. With minor modifications (using a bootstrap argument), the proof also yields modified scattering provided $|\lambda|\| \la v\ra^3\mu_0\|_{L^{\infty}}^2$ is sufficiently small  when $\lambda<0$. See Remark \ref{rem:attractive}.
 \item The asymptotic behavior in \eqref{eq:mscatter} exhibits modified scattering with a power-law correction in the lens variables (see \eqref{eq:lens}). More precisely, the momentum correction is $\lambda\beta(s)E_\infty(q)$, where
    \[
        \beta(s)=\frac{s}{\sqrt{1-s^2}}\sim\frac{1}{\sqrt{2(1-s)}}\qquad\text{as }s\uparrow1.
    \]
The corresponding correction for the two-dimensional problem studied in \cite{BVR25,HK26} grows logarithmically in the lens variables as $s\uparrow1$. In the original variables, the corresponding trajectories for the two-dimensional
problem \cite{BVR25,HK26} are
    \begin{equation*}
        \begin{split}
  \mathcal{X}(t,x,v)&=x\cosh t+v\sinh t-\lambda te^{-t}E_{\oo}(x+v),\\
\mathcal{V}(t,x,v)&=v\cosh t+x\sinh t+\lambda te^{-t}E_{\oo}(x+v).
        \end{split}
    \end{equation*}
In the stable coordinate $x-v$, free trajectories satisfy $(x-v)(t)\to0$. In two dimensions the corrected trajectories still satisfy $\mathcal X-\mathcal V\to0$, only at the slower rate $te^{-t}$, whereas in one dimension $\mathcal X-\mathcal V\to-\lambda E_\infty(x+v)$: the interaction leaves a permanent
displacement.

\end{enumerate}
\end{remark}

\smallskip

Our second result establishes the existence of the modified wave operator.

\smallskip

\begin{theorem}\label{thm:main2}
Assume $\lambda\neq0$. Let $\mu_\infty$ satisfy
\[
\sum_{j+k\le2}\sum_{r\in\{2,\infty\}}
\bigl\|\langle p\rangle^3
\partial_q^j\partial_p^k\mu_\infty\bigr\|_{L^r}
<\infty.
\]
Set $E_\infty=E[\mu_\infty]$. There exists an initial datum $\mu_0$ satisfying \eqref{eq:data} such that the Cauchy problem \eqref{eq:VPH} with initial condition $\mu(0)=\mu_0$
admits a unique global classical solution $\mu$.
Moreover, $\mu$ satisfies \eqref{eq:mscatter}
with the prescribed profile $\mu_\infty$.
\end{theorem}

\begin{remark}\leavevmode
\begin{enumerate}[label=(\arabic*),leftmargin=*,topsep=0.3em,itemsep=0.3em]
\item Theorem \ref{thm:main2} holds for both signs of $\lambda$ and for data of arbitrary size. Since \eqref{eq:VPH} is invariant under $(t,x,v)\mapsto(-t,x,-v)$, it also gives modified wave operators as $t\to-\infty$.
\item Combining this with Theorem \ref{thm:main}, a profile $\mu_{-\oo}$ in the class of Theorem~\ref{thm:main2} is mapped to $\mu_0$ and then to $\mu_{+\oo}$. This defines a scattering map for data of arbitrary size if $\lambda>0$, and, for $\lambda<0$, whenever $\mu_0$ satisfies the smallness condition of Remark \ref{rem:attractive}. The map is not between the same spaces: Theorem~\ref{thm:main} only gives $\mu_{+\oo}\in L^2\cap L^\infty$, while Theorem \ref{thm:main2} requires two weighted derivatives.
\item Unlike \cite{HK26}, the convergence here is uniform with an explicit rate.

\end{enumerate}   
\end{remark}

\subsection{Difficulties and ideas of the proof}

In low dimensions, much less is known about the long-time behavior of the Vlasov--Poisson system; see \cite{Sch07,GPS10} for results in one dimension without external potential. Under free transport the density disperses at the rate $t^{-d}$, so that the self-consistent field decays like $t^{-(d-1)}$. For $d\ge4$ this leads to linear scattering \cite{CH06}, while $d=3$ is scattering-critical and requires a logarithmic correction along the characteristics \cite{CK16,IPWW22,FOPW23}. In $d=2$ the field decays only like $t^{-1}$, so that already the velocity correction diverges logarithmically, and in $d=1$ the field does not decay at all: $E[\mu]\to\pm M/2$ at spatial infinity, and particles keep being accelerated for all times.  Even with screened interactions the one-dimensional problem is critical: free scattering holds for $d\ge2$ \cite{IaRoKl26}, whereas for $d=1$ the screened force still decays only like $t^{-2}$, and Wei \cite{Wei25} proved sharp decay of the density for small Gevrey data. 

\medskip
A repulsive harmonic potential changes this picture. The unstable trapping at the origin separates trajectories exponentially, so that the density decays like $e^{-dt}$. In the lens variables of \cite{HK26}, the self-consistent force then carries the time weight $(1-s^2)^{-(4-d)/2}$, which is integrable for $d\ge3$ (linear scattering \cite{VR24,HK26} ) and logarithmically divergent for $d=2$ (modified scattering with a logarithmic correction \cite{BVR25,HK26}). For $d=1$ the weight is $(1-s^2)^{-3/2}$: it is not integrable, but its primitive is, so the problem becomes accessible while the correction is of power type. In the original variables, this correction is a persistent $O(1)$ displacement along the stable direction. Moreover, for $\lambda>0$, the one-dimensional identity $\partial_xE=\rho\ge0$  allows us to treat data of arbitrary size.

\subsubsection*{Control of moments}
We first use the hyperbolic lens transform introduced in \cite{HK26}, which sends $t=\oo$ to the finite endpoint $s=1$ and converts \eqref{eq:VPH} into
\[
\pr_s\gamma+p\pr_q\gamma+\lambda\alpha(s)E[\ga]\pr_p\gamma=0,
    \qquad \alpha(s)=(1-s^2)^{-3/2}.
\]

Unlike the two-dimensional problem, where the moments grow at most logarithmically, in one dimension the boundedness of $E$ only yields polynomial (in $(1-s)^{-1}$) growth of the momentum moments of $\ga$:
\begin{equation*}
    \|\la p\ra^a \ga(s)\|_{L^r}\les (1-s)^{-a/2},\qquad r\in\{2,\infty\}.
\end{equation*}
Hence, by the continuity equation,
\begin{equation*}
    \pr_s E(s,q)=-\int_\R p\,\ga^2(s,q,p)\dd p,    \qquad   |\pr_s E(s,q)|\le \|\la p\ra^a\ga(s)\|_{L^\infty}^2\int_\R\la p\ra^{1-2a}\dd p     \les (1-s)^{-a},
\end{equation*}
where $a>1$ is needed for the $p$-integral to converge. This bound is not integrable near $s=1$ for any admissible $a$, so the direct estimates do not  show that $E(s)$ has a limit as $s\to1$. 

Inspired by the first-order correction introduced in \cite{HX26} for long-range Riesz interactions, we remove the singular force by a momentum correction. Let $\Theta$ solve the Hamilton--Jacobi equation
\[
\pr_s\Theta+\tfrac12(\pr_q\Theta)^2=\lambda\alpha(s)\phi(s,q),\qquad \Theta(0,q)=0, \qquad \pr_q\phi=E,
\]
so that $\Gamma=\pr_q\Theta$ solves the forced Burgers equation
\begin{equation}\label{eq:Burg-intro}
\pr_s\Gamma+\Gamma\pr_q\Gamma=\lambda\alpha(s)E,\qquad \Gamma(0,q)=0.
\end{equation}
The generating function $S(s,q,p)=qp+\Theta(s,q)$ transforms the lens Hamiltonian $\frac{p^2}{2}-\lambda\alpha(s)\phi$ into
\[
\K=\frac{p^2}{2}+p\,\pr_q\Theta +\Big(\pr_s\Theta+\tfrac12(\pr_q\Theta)^2-\lambda\alpha(s)\phi\Big)=\frac{p^2}{2}+p\Gamma,
\]
so that $\sigma(s,q,p)=\gamma(s,q,p+\Gamma(s,q))$ satisfies
\[
\pr_s\sigma+\{\sigma,\K\}=0.
\]

Heuristically, $\Gamma$ is a single time integral of the singular force: along the
characteristics of \eqref{eq:Burg}, it is obtained by integrating $\lambda\alpha(s)E$
once, so that
\begin{equation*}
    \|\Gamma(s)\|_{L^\infty}\lesssim\int_0^s\alpha(\tau)\dd\tau=\beta(s)\simeq(1-s)^{-1/2},
\end{equation*}
and the same bound holds for its spatial derivatives. As a consequence,  the vector fields associated with  the Hamiltonian $\mathcal{K}$  are integrable. Hence  the moments of $\sigma$  are  bounded.  Derivatives are more delicate: the equation for $\pr_q\sigma$ contains $p\,(\pr_q^2\Gamma)\,\pr_p\sigma$, which costs one power of $p$, and we estimate $\pr_q\sigma$ and $\pr_p\sigma$ together in Proposition~\ref{prop:der}.

\subsubsection*{Large data bounds}
To deal with large data, we observe that the singular force has been replaced by a term whose sign can be exploited. Indeed, the Burgers equation for $\Gamma$ gives
\[
    (\pr_s+\Gamma\pr_q)\pr_q\Gamma+(\pr_q\Gamma)^2
    =\lambda\alpha(s)\int_\R\sigma^2(s,q,p)\dd p\ge0.
\]
For $\lambda>0$, the method of characteristics and an ODE argument yield $\pr_q\Gamma\ge0$. This sign produces momentum damping: the norms $\|\la p\ra^3\sigma\|_\infty$ and $\|\la p\ra\sigma\|_2$ are nonincreasing. These estimates give a uniform density bound and, in turn, bounds for $\Gamma$ and $\pr_q\Gamma$ that ensure continuation of the correction for every $s<1$. The sign of $\pr_q\Gamma$, rather than smallness, closes these estimates. For $\lambda<0$ one needs a bootstrap instead, see Remark~\ref{rem:attractive}.

\subsubsection*{Asymptotic behavior}
In terms of $\sigma$, the field equation gives $\|\pr_sE(s)\|_{L^\infty}\les1+\beta(s)$, so $E(s)$ converges uniformly to some $E_\infty$, at rate $\sqrt{1-s^2}$. We then follow $\gamma$ along the  trajectories as in \cite{HK26}
\[
 \Phi^1(s,q,p)=q-(1-s)\big(p+\lambda\beta(s)E_\infty(q)\big),\qquad \Phi^2(s,q,p)=p+\lambda\beta(s)E_\infty(q),
\]
which carry the momentum shift generated by $E_\infty$ and reach $q$ at $s=1$. Setting $\nu(s,q,p)=\gamma(s,\Phi^1,\Phi^2)$, the goal is to show that $\pr_s\nu$ is integrable on $[0,1)$. A direct computation gives
\[
    \pr_s\nu =\underbrace{-\lambda(1-s)\alpha(s)E_\infty(q)\,\pr_q\gamma(s,\Phi)}_{\rm (I)}+\underbrace{\lambda\alpha(s)\big[E_\infty(q)-E(s,\Phi^1)\big]\pr_p\gamma(s,\Phi)}_{\rm (II)}.
\]
Since $(1-s)\alpha(s)\simeq(1-s)^{-1/2}$ is integrable while $\alpha(s)$ is not, we need $\pr_q\gamma$ to be bounded in (I), and the field error in (II) to be of order $1-s$. The bounds for $\sigma$ alone give neither. Since $\pr_q\gamma=(\pr_q\sigma-\pr_q\Gamma\,\pr_p\sigma)(s,q,p-\Gamma)$ and $\pr_q\Gamma\simeq\beta(s)$, the term (I) is of size $(1-s)\alpha\beta\simeq(1-s)^{-1}$. In (II), both $E(s)-E_\infty$ and $\Phi^1-q$ are only $\cO(\sqrt{1-s})$, so (II) is also of size $(1-s)^{-1}$. Each term misses integrability by exactly a factor
$(1-s)^{-1/2}$.

The two losses are related, and we remove them together. The growing part of $\pr_q\gamma$ is $-\pr_q\Gamma\,\pr_p\gamma\approx-\lambda\beta\,\pr_qE\,\pr_p\gamma$. Moving it from (I) to (II), we rewrite
\[
\begin{aligned}
    \pr_s\nu={}&-\lambda(1-s)\alpha E_\infty(q) \big[\pr_q\gamma+\lambda\beta\,\pr_qE\,\pr_p\gamma\big](s,\Phi)\\
    &+\lambda\alpha\big[E_\infty(q)-E(s,\Phi^1)+\lambda(1-s)\beta E_\infty(q)\,\pr_qE(s,\Phi^1)\big]\pr_p\gamma(s,\Phi).
\end{aligned}
\]
It therefore suffices to show:
\begin{enumerate}[label=(\alph*),leftmargin=*]
\item the first bracket is bounded;
\item the second bracket is $\cO\big((1-s)\la p\ra\big)$;
\item $|\pr_p\gamma(s,\Phi)|\les\la p\ra^{-1}$.
\end{enumerate}
Property (c) follows from the weighted bound for $\la p\ra\pr_p\sigma$. For (a) and (b), we compare the nonlinear correction with its explicit leading part,
\[
B(s,q)=\Gamma(s,q)-\lambda\beta(s)E(s,q).
\]
In the Burgers equation for $\Gamma$, the singular forcing $\lambda\alpha E$ is cancelled exactly by the time derivative of $\lambda\beta E$, since $\beta'=\alpha$. Hence $B$ and $\pr_qB$ satisfy transport equations with integrable sources and remain bounded. This gives (a), because
\[
\pr_q\gamma+\lambda\beta\,\pr_qE\,\pr_p\gamma =(\pr_q\sigma-\pr_qB\,\pr_p\sigma)(s,q,p-\Gamma).
\]
It also gives one more order in the field. Substituting $\Gamma=B+\lambda\beta E$ into $\pr_sE=-\int_\R p\,\sigma^2\dd p-\Gamma\,\pr_qE$, the terms involving $\sigma$ and $B$ are bounded, while $-\lambda\beta E\,\pr_qE$ has size $\beta(s)$. Integrating from $s$ to $1$ yields
\[
  E(s,q)=E_\infty(q)+\lambda\sqrt{1-s^2}\,E_\infty(q)\pr_qE_\infty(q)+\cO(1-s).
\]
Since $\Phi^1(s,q,0)-q=-\lambda(1-s)\beta(s)E_\infty(q)$, a Taylor expansion shows that, at $p=0$, the second bracket equals
\[
    \lambda\big[2(1-s)\beta(s)-\sqrt{1-s^2}\big]E_\infty\pr_qE_\infty+\cO(1-s) =-\lambda\frac{(1-s)^2}{\sqrt{1-s^2}}E_\infty\pr_qE_\infty+\cO(1-s).
\]
That is, the $\sqrt{1-s}$ correction of the field cancels exactly against the shift of the spatial argument. For general $p$, $\Phi^1$ moves by $(1-s)p$, which costs $\cO((1-s)|p|)$. This proves (b). Combining (a)--(c),
\[
  \|\pr_s\nu(s)\|_{L^\infty}\les(1-s)\alpha(s),\qquad  \int_s^1(1-\tau)\alpha(\tau)\dd\tau=\sqrt{\tfrac{1-s}{1+s}}=e^{-t} \quad (s=\tanh t),
\]
so $\nu$ converges uniformly as $s\uparrow1$, and returning to the original variables gives Theorem~\ref{thm:main} with the rate $e^{-t}$.

\subsubsection*{Modified wave operator}
\smallskip
To construct the wave operator we solve \eqref{eq:lensP} backward from $s=1$. The
Burgers correction cannot be used for this: $\Gamma$ is defined by integrating
\eqref{eq:Burg-intro} forward from $s=0$, so it is not available at $s=1$. We therefore
shift the momentum by its explicit leading part $\lambda\beta(s)E(s,q)$, which involves
only the field at the same time. The resulting equation \eqref{eq:tK} has coefficients of
size $\beta(s)$, which is integrable, but the term $\lambda\beta\,p\,\pr_q^2E$ costs one
power of $p$ in the derivative estimates. We recover it with a vector field
$Z=A\pr_q+C\pr_p$ that commutes with the equation. By the continuity equation, the second
derivative of $A$ along the flow equals $\lambda\alpha(s)\,\pr_qE\,A$, so after two time
integrations only the integrable kernel $(1-s)\alpha(s)$ appears, as in the forward
problem.

\begin{remark}[A comparison with $3d$ Vlasov-Riesz system]
If we consider the three dimensional Vlasov-Riesz system,
\begin{equation*}
    \pr_t\mu+v\cdot \nab_x\mu+\lambda E[\mu]\cdot \nab_v\mu=0,\quad E=\nabla\phi=\nab(|x|^{-a}*\rho).
\end{equation*}
Using pseudo-conformal compactification, one obtains
\begin{equation*}
    \begin{split}
        \pr_s\ga+p\cdot \nab_q\ga+\lambda s^{a-2}E[\ga]\cdot \nab_p\ga=0.
    \end{split}
\end{equation*}
\cite{HP26} proves modified scattering for $1/2<a<1$, see also \cite{HX26}. When $a=1/2$, the singularity is the same order as in \eqref{eq:lensP}. Therefore, we expect our methods can be adapted to  the Vlasov-Riesz system in the very long range case $0<a\leq 1/2$. This is addressed in the work with M. Xie \cite{HX26-2}.
\end{remark}

\subsection{Organization}
In Section \ref{sec:pre}, we introduce the lens transform and collect preliminary
estimates. In Section \ref{sec:correction}, we introduce the Hamilton--Jacobi phase
correction and establish moment and derivative bounds for the corrected
distribution. Section \ref{sec:electric-field} is devoted to the asymptotic behavior of the
electric field. In Section \ref{sec:main}, we prove modified scattering and complete
the proof of the main theorem.  In
Section \ref{sec:wave}, we prove the existence of the modified wave operators. Appendix \ref{app:local} provides the finite-time
existence and continuation arguments used in the proof.

\section{Preliminaries}\label{sec:pre}
In this section, we collect the analytic tools used throughout the paper. We first introduce the lens transform and then recall the basic estimates for transport equations.
\subsection{Lens transform}

We use the lens transform introduced in \cite{HK26}
\begin{equation}\label{eq:lens}
 s=\tanh t,\qquad q=\frac{x}{\cosh t},\qquad
 p=v\cosh t-x\sinh t,\qquad \mu(t,x,v)=\gamma(s,q,p),
\end{equation}
which is symplectic. The equation for $\ga$ becomes
\begin{equation}\label{eq:lensP}
 \partial_s\gamma+p\pr_q\gamma+\lambda (1-s^2)^{-3/2}E[\ga](s,q)\pr_p\gamma=0,
 \qquad \rho(s,q)=\pr_qE(s,q)=\int_\R\gamma^2(s,q,p)\dd p,
\end{equation}
which can be rewritten in the Hamiltonian form
\begin{equation}\label{eq:lensH}
    \begin{split}
        \pr_s\ga+\{\ga,\H\}=0,\quad \H=\frac{p^2}{2}-\lambda\alpha(s)\phi(s,q),\quad \pr_q\phi=E.
    \end{split}
\end{equation}
For convenience, we write
\begin{equation}\label{eq:fF}
 \alpha(s)=(1-s^2)^{-3/2},\qquad \beta(s)=\frac{s}{\sqrt{1-s^2}},\qquad \beta'=\alpha.
\end{equation}
The initial condition is $\gamma(0,q,p)=\mu_0(q,p)$.  We note that $\beta(s)\sim (1-s)^{-1/2}$ is integrable near $s=1$ and $\alpha(s)\sim (1-s)^{-3/2}$ near $s=1$.

\subsection{Transport estimates}
The first frequently used lemmas follow from the method of characteristics.
\begin{lemma}\label{lem:Hamilton}
Suppose that $r\in[1,\infty]$ and
$\gamma\in C([0,T];L^r_{q,p})$ is a solution to
\[    \partial_s\gamma+\{\gamma,\mathcal{K}\}=g,    \qquad \gamma(s=0)=\gamma_0
\]
for some $\gamma_0\in L^r_{q,p}$. We assume that the Hamiltonian vector field is locally continuously differentiable in space and has at most linear growth uniformly on $[0,T]$. Then we have
\[
    \|\gamma(s)\|_{L^r_{q,p}}
    \le \|\gamma_0\|_{L^r_{q,p}}
    +\int_0^s\|g(\tau)\|_{L^r_{q,p}}\dd\tau,
    \qquad s\in[0,T].
\]
Moreover, if $g=\ga \cdot F+G$ where $F\leq 0$, then we have
\begin{equation*}
    \begin{split}
        \|\ga(s)\|_{L^r}\leq \|\ga_0\|_{L^r}+\int_0^s \|G(\tau)\|_{L^{r}}d\tau \quad s\in[0,T].
    \end{split}
\end{equation*}

\end{lemma}
\begin{proof}
Let $\Phi_s$ denote the Hamiltonian flow starting at time zero:
\[
\frac{\dd}{\dd s}\Phi_s(q,p)
=(\pr_p\mathcal K,-\pr_q\mathcal K)(s,\Phi_s(q,p)),
\qquad \Phi_0(q,p)=(q,p).
\]
The regularity and growth assumptions ensure that $\Phi_s$ is a globally defined $C^1$ diffeomorphism. Since the Hamiltonian vector field is divergence free, Liouville's formula gives
\[
\det D\Phi_s=1.
\]
Consequently, composition with $\Phi_s$ preserves every $L^r$ norm, including $r=\infty$.

Along the flow, the equation becomes
\[
\frac{\dd}{\dd s}\gamma(s,\Phi_s(q,p))=g(s,\Phi_s(q,p)).
\]
Integration therefore yields
\[
\gamma(s,\Phi_s(q,p))=\gamma_0(q,p)+\int_0^s g(\tau,\Phi_\tau(q,p))\dd\tau.
\]
Taking the $L^r_{q,p}$ norm and using Minkowski's inequality and preservation of measure, we obtain
\[
\begin{aligned}
\|\gamma(s)\|_{L^r}
&=\|\gamma(s)\circ\Phi_s\|_{L^r}\\
&\le\|\gamma_0\|_{L^r}+\int_0^s\|g(\tau)\circ\Phi_\tau\|_{L^r}\dd\tau\\
&=\|\gamma_0\|_{L^r}+\int_0^s\|g(\tau)\|_{L^r}\dd\tau.
\end{aligned}
\]

If $g=F\gamma+G$ with $F\le0$, the integrating-factor formula along the same flow gives
\[
\begin{aligned}
\gamma(s,\Phi_s(q,p))
&=e^{\int_0^s F(\tau,\Phi_\tau(q,p))\dd\tau}\gamma_0(q,p)\\
&\quad+\int_0^s e^{\int_\tau^s F(u,\Phi_u(q,p))\dd u} G(\tau,\Phi_\tau(q,p))\dd\tau.
\end{aligned}
\]
Both exponential factors are bounded by one. Hence
\[
|\gamma(s,\Phi_s(q,p))|\le|\gamma_0(q,p)|+\int_0^s|G(\tau,\Phi_\tau(q,p))|\dd\tau.
\]
Taking the $L^r$ norm as above proves the second estimate.
\end{proof}

To estimate the correction $\Gamma$ and its derivatives in \eqref{eq:Burg} , we also use the following $L^\infty$ estimate for transport equations with a zeroth-order term. Unlike in Lemma \ref{lem:Hamilton}, the transport field need not be divergence free.

\begin{lemma}\label{lem:prop}
    Let \(u\) be a bounded continuous distributional solution of
\[
\partial_su+b(s,z)\cdot\nabla_z u+c(s,z)u=g(s,z)
\]
on \([0,T]\times\mathbb R^d\).  Assume that $b\in L^1((0,T);C^1_{\mathrm{loc}}(\R^d;\R^d))$ and  $c,g\in L^1([0,T);L^{\oo}(\R^d))$ satisfy
\[
|b(s,z)|\le K(1+|z|),\qquad c(s,z)\ge-a(s),\qquad |g(s,z)|\le G(s),
\]
where \(a,G\) are nonnegative continuous functions. Then
\[
{\|u(s)\|_\infty\le e^{\int_0^s a(r)\,dr}\|u(0)\|_\infty+\int_0^s e^{\int_r^s a(v)\,dv}G(r)\,dr.}
\]
In particular, if \(c\ge0\), then
\[
\|u(s)\|_\infty \le \|u(0)\|_\infty+\int_0^sG(r)\,dr.
\]
\end{lemma} 
\begin{proof}
Fix $s\in[0,T]$ and $z\in\R^d$. Consider a characteristic satisfying
\[
\frac{\dd}{\dd\tau}X(\tau)=b(\tau,X(\tau)), \qquad X(s)=z.
\]
The continuity and linear growth of $b$ ensure that such a characteristic exists on $[0,s]$. Indeed, Gronwall's inequality gives
\[
1+|X(\tau)|\le (1+|z|)e^{K(s-\tau)}, \qquad 0\le\tau\le s,
\]
which prevents finite-time escape to infinity. 

Along this curve, the chain rule gives
\[
\frac{\dd}{\dd\tau}u(\tau,X(\tau))+c(\tau,X(\tau))u(\tau,X(\tau))=g(\tau,X(\tau)).
\]
Multiplying by an integrating factor and integrating from $0$ to $s$, we obtain
\[
\begin{aligned}
u(s,z)
&=e^{-\int_0^s c(r,X(r))\dd r}\,u(0,X(0))\\
&\quad+\int_0^s e^{-\int_\tau^s c(r,X(r))\dd r} g(\tau,X(\tau))\dd\tau.
\end{aligned}
\]
Since $c(r,X(r))\ge-a(r)$ and $|g(\tau,X(\tau))|\le G(\tau)$, it follows that
\[
|u(s,z)|\le e^{\int_0^s a(r)\dd r}\|u(0)\|_\infty+\int_0^s e^{\int_\tau^s a(r)\dd r}G(\tau)\dd\tau.
\]
Taking the supremum over $z\in\R^d$ proves the claimed estimate. If $c\ge0$, we may take $a=0$, yielding
\[
\|u(s)\|_\infty\le\|u(0)\|_\infty+\int_0^sG(\tau)\dd\tau.
\]
\end{proof}

\subsection{Notation}
For $1\le r\le\infty$, we abbreviate $L^r_{q,p}$ by $L^r$ when the underlying variables are clear from context. If a function depends on only one of the phase-space variables, the norm is taken over that variable alone. We also write $\|\cdot\|_r$ in place of $\|\cdot\|_{L^r}$. Unless otherwise specified, all these norms are
evaluated at a fixed time.

We write $A\lesssim B$ if $A\le CB$ for some constant $C>0$. The notation $\cO(h(s))$ denotes a quantity bounded in absolute value by $C|h(s)|$, uniformly in the remaining variables and in the limiting regime under consideration. Constants may depend on $C_0$ and $\lambda$, but are independent of time and may vary from line to line.

\section{A Hamilton--Jacobi phase correction}\label{sec:correction}
Throughout Sections \ref{sec:correction}-\ref{sec:main}, we assume $\lambda>0$. The finite-time theory in Appendix \ref{app:local} and Section \ref{sec:wave} allow either sign of $\lambda$. We begin with the Hamilton--Jacobi equation
\begin{equation*}
    \begin{split}
        \pr_s\Theta +\frac{1}{2}(\nab_q\Theta)^2=\lambda\alpha(s)\phi(s,q), \quad \Theta(0,q)=0.
    \end{split}
\end{equation*}
We note here a similar correction has appeared in \cite{GV00I,GV00II} before. Differentiating in $q$ and writing $\Gamma(s,q)=\nab_q\Theta$, we obtain the forced Burgers equation
\begin{equation}\label{eq:Burg}
    \begin{split}
\pr_s\Gamma+\Gamma\pr_q\Gamma=\lambda \alpha(s)E(s,q), \qquad \Gamma(0,q)=0.
    \end{split}
\end{equation}
The generating function $S(s,q,p)=qp+\Theta(s,q)$ gives the canonical symplectic transformation
\begin{equation}\label{eq:sigma}
    \begin{split}
\sigma(s,q,p):=\ga(s,q,p+\Gamma(s,q)),
    \end{split}
\end{equation}
which satisfies
\begin{equation}\label{eq:sigma-K}
    \begin{split}
        \pr_s\sigma+\{\sigma,\K \}=0, \quad  \K=\frac{p^2}{2}+p\Gamma,
    \end{split}
\end{equation}
where $\{f,g\}=\pr_q f\cdot \pr_p g-\pr_p f\cdot \pr_q g$. All estimates below are first proved on the common local existence interval of $\gamma$ and $\Gamma$. The construction and continuation criteria are supplied in Proposition~\ref{prop:time}.

\begin{proposition}\label{prop:mom}
Under \eqref{eq:data}, the correction $\Gamma$ exists for every $s<1$, and 
\begin{equation}\label{eq:mom}
    \begin{split}
       \|\langle p\rangle^3\sigma(s)\|_{L^{\oo}} \leq C_0, \quad      \|\langle p\rangle \sigma(s)\|_{L^2}\leq C_0,\\
       0\leq \pr_q\Gamma\les \beta(s), \quad \|\Gamma(s)\|_{L^{\oo}}\les\beta(s).
    \end{split}
\end{equation}
As a consequence,
\begin{equation}\label{eq:dE}
    \begin{split}
        \|\pr_q E(s)\|_{L^{\infty}}=\|\rho(s)\|_{L^{\oo}}
        \les  1.
    \end{split}
\end{equation}

\end{proposition}
\begin{proof}
We first establish $\pr_q\Gamma\geq0$. Differentiating \eqref{eq:Burg} gives
\begin{equation}\label{eq:dBurg}
    \begin{split}
        \pr_s (\pr_q\Gamma)+\Gamma\pr_q(\pr_q\Gamma)=\lambda \alpha(s)\pr_q E(s,q)-(\pr_q\Gamma)^2, \qquad \pr_q\Gamma(0,q)=0,
    \end{split}
\end{equation}
where $\pr_q E=\rho\geq0$. Consider the characteristic ODE for fixed $q\in \R$, 
\begin{equation*}
    \begin{split}
        \frac{d}{ds}Q(s)=\Gamma(s,Q(s)),\quad Q(0)=q.
    \end{split}
\end{equation*}
Let $a(s):=\pr_q\Gamma(s,Q(s))$, we have
\begin{equation*}
    \begin{split}
        a^\prime(s)+a^2(s)=\lambda \alpha(s) \rho(s,Q(s))\geq 0,\quad a(0)=0.
    \end{split}
\end{equation*}
Therefore, we have
\begin{equation*}
    \begin{split}
        \frac{d}{ds}\off{a(s)\exp\of{\int_0^s a(\tau)\dd\tau} }=\lambda \alpha(s)\rho(s,Q(s))\exp \of{\int_0^s a(\tau)\dd\tau}\geq 0.
    \end{split}
\end{equation*}
This proves $\pr_q\Gamma\geq 0$. We next estimate the velocity moments. For $m\ge0$,
\begin{equation*}
    \pr_s(\la p\ra^m\sigma)+\{\la p\ra^m\sigma,\K\}=\sigma\{\la p\ra^m,\K\}    =-m(\pr_q\Gamma)\frac{p^2}{1+p^2}\la p\ra^m\sigma .
\end{equation*}
Since $\pr_q\Gamma\ge0$, the coefficient $-m(\pr_q\Gamma)p^2/(1+p^2)$ is nonpositive. As $\sigma(0)=\mu_0$, the second part of Lemma~\ref{lem:Hamilton} (with $G=0$) gives
\begin{equation*}
    \|\la p\ra^m\sigma(s)\|_{L^r}\le\|\la p\ra^m\mu_0\|_{L^r},    \qquad m\ge0,\ r\in[1,\infty].
\end{equation*}
Taking $(m,r)=(3,\infty)$ and $(m,r)=(1,2)$ proves the first two estimates in \eqref{eq:mom}.
The velocity translation in \eqref{eq:sigma} leaves the density unchanged, so
\[
 \rho(s,q)=\int_\R\sigma^2(s,q,p)\dd p \leq C_0^2\int_\R(1+p^2)^{-3}\dd p  \les 1.
\]
This proves \eqref{eq:dE}. Applying Lemma \ref{lem:prop} with $u=\pr_q\Gamma, b=\Gamma$ and $c=\pr_q\Gamma\geq 0$ to \eqref{eq:dBurg}  gives
$$0\leq\pr_q\Gamma\les \beta(s).$$
Likewise, \eqref{eq:Burg} and $\|E\|_\infty\leq M/2$ give $\|\Gamma(s)\|_\infty\leq\lambda M\beta(s)/2$.

\smallskip
Finally, for every $T<1$, the bounds for $\Gamma$ and $\pr_q\Gamma$ are uniform on $[0,T]$, the continuation criterion in Proposition \ref{prop:time} therefore extends the  correction to every $s<1$.

\end{proof}

\begin{remark}[Small data with $\lambda<0$]\label{rem:attractive}
In Sections \ref{sec:correction}--\ref{sec:main}, the sign $\lambda>0$ enters only through the inequality $\pr_q\Gamma\ge0$. It is used to rule out shocks in \eqref{eq:Burg}, and to discard the zeroth-order terms involving $\pr_q\Gamma$ when Lemmas \ref{lem:Hamilton} and \ref{lem:prop} are applied. When $\lambda<0$, the equation along the characteristics of \eqref{eq:Burg},
\begin{equation}\label{eq:Ric}
    \frac{\dd a}{\dd s}+a^2=\lambda\alpha(s)\rho\big(s,Q(s)\big),    \qquad a(s)=\pr_q\Gamma\big(s,Q(s)\big),
\end{equation}
is focusing and this sign is lost. For small data it can be replaced by a bootstrap. Set $\delta:=|\lambda|\,\|\la v\ra^3\mu_0\|_{L^\oo}^2$ and assume that, on some interval $[0,T]$ with $T<1$,
\begin{equation}\label{eq:boot-attr}
    \|\pr_q\Gamma(s)\|_{L^\oo}\le K\delta\,\beta(s),\qquad 0\le s\le T .
\end{equation}
Since $\int_0^1\beta(s)\dd s=1$, the argument in the proof of Lemma \ref{lem:Hamilton} gives $\|\la p\ra^m\sigma(s)\|_{L^r}\le e^{mK\delta}\|\la p\ra^m\mu_0\|_{L^r}$. Hence
\[
|\lambda|\,\|\rho(s)\|_{L^\oo}\le c_1e^{6K\delta}\delta,\qquad c_1=\int_\R\la p\ra^{-6}\dd p.
\]
Integrating \eqref{eq:Ric}  and using
\begin{equation*}
\int_0^s\beta(\tau)^2\dd\tau=\operatorname{artanh}s-s\le\beta(s),
\end{equation*}
we obtain
\begin{equation*}
\|\pr_q\Gamma(s)\|_{L^\oo}\le\int_0^s|\lambda|\alpha(\tau)\|\rho(\tau)\|_{L^\oo}\dd\tau+\int_0^s\|\pr_q\Gamma(\tau)\|_{L^\oo}^2\dd\tau\le\big(c_1e^{6K\delta}+K^2\delta\big)\,\delta\,\beta(s).
\end{equation*}
With $K=4c_1$ and $\delta$ sufficiently small, this improves \eqref{eq:boot-attr} to $\frac K2\delta\beta(s)$. The continuation criterion in Proposition \ref{prop:time} then shows that $\Gamma$ exists on $[0,1)$ and that \eqref{eq:boot-attr} holds there.

In the remaining estimates (Proposition~\ref{prop:der}, Section~\ref{sec:electric-field} and Section~\ref{sec:main}), $\pr_q\Gamma\ge0$ is used only to give the zeroth-order coefficients a favorable sign. Replacing it by $|\pr_q\Gamma|\le K\delta\beta(s)$ and using $\int_0^1\beta=1$ only changes the constants by a factor $e^{CK\delta}$. Consequently, Theorem~\ref{thm:main} also holds for $\lambda<0$, provided $|\lambda|\,\|\la v\ra^3\mu_0\|_{L^\oo}^2$ is sufficiently small. 
\end{remark}

Next, we bound the derivatives.

\begin{proposition}\label{prop:der}
    Under \eqref{eq:data}, the following estimates hold for $0\leq s<1$:
    \begin{equation}
       \begin{split}
           \|\langle p\ra\pr_p\sigma\|_{L^{\oo}}+\|\pr_q\sigma(s)\|_{L^{\oo}}&\leq C, \\
           \|\la p\ra \pr_q\sigma\|_{L^{\oo}}&\leq C(1+\beta(s)).
       \end{split} 
    \end{equation}
Moreover,
\begin{equation}
    \begin{split}
     \|\nab_q^2\Gamma(s)\|_{L^{\oo}}\leq C\beta(s),\quad     \|\pr_q^2 E(s)\|_{\infty}\leq C.
    \end{split}
\end{equation}

\end{proposition}
\begin{proof}
For convenience, write
\begin{equation*}
    \begin{split}
        \D f:=\pr_sf+\{f,\K\}.
    \end{split}
\end{equation*}
As in Proposition \ref{prop:mom}, we can compute
\begin{equation}\label{eq:D1}
    \begin{split}
        \D (\pr_q\sigma)&=p(\pr_q^2\Gamma)(\pr_p\sigma)-\pr_q\Gamma\,\pr_q\sigma,\\
        \D(\pr_p\sigma)&=\pr_q\Gamma\, \pr_p\sigma-\pr_q\sigma,\\
        \D(\pr_q\Gamma)&=\lambda\alpha(s)\rho(s,q)-(\pr_q\Gamma)^2+p\cdot \pr_q^2\Gamma.
    \end{split}
\end{equation}
As a consequence, we have
\begin{equation}\label{eq:D2}
    \begin{split}
        \D(\pr_q\sigma-(\pr_q\Gamma)\pr_p\sigma)&=-\lambda \alpha(s)\rho(s,q)\pr_p\sigma,\\
        \D(\pr_p\sigma)&=\pr_q\Gamma \, \pr_p\sigma-\pr_q \sigma.
    \end{split}
\end{equation}
With $Z(s,q,p)=\pr_q\sigma-\pr_q\Gamma\cdot\pr_p\sigma$ and $W(s,q,p)=\pr_p\sigma$, these equations become
\begin{equation}\label{eq:ZW}
    \begin{split}
      \D Z&=-\lambda \alpha(s)\rho W,\\
      \D W&=-Z.
    \end{split}
\end{equation}
Multiplying by the velocity weight gives
\begin{equation}
    \begin{split}
\D (\la p\ra Z)&=[\D \langle p\ra]Z-\lambda \alpha(s)\rho \la p\ra W,\\
\D (\la p\ra W )&=[\D\la p\ra]W-\la p\ra Z.
    \end{split}
\end{equation}
The weight satisfies
\begin{equation*}
    \begin{split}
      \frac{\D\la p\ra}{\la p\ra}      =-\frac{p^2}{1+p^2}\pr_q\Gamma\leq0,\qquad      \frac{|\D\la p\ra|}{\la p\ra}\les\beta(s).
    \end{split}
\end{equation*}
Thus the resulting zeroth-order terms have a favorable sign.
Lemma~\ref{lem:prop} and \eqref{eq:dE} give
\[
\begin{split}
 \|\la p\ra W(s)\|_\infty &\leq C_0+\int_0^s\|\la p\ra Z(\tau)\|_\infty\dd\tau,\\
 \|\la p\ra Z(s)\|_\infty &\leq C_0+C\int_0^s\alpha(\tau)
                         \|\la p\ra W(\tau)\|_\infty\dd\tau.
\end{split}
\]
Substitution and interchange of the two nonnegative integrals yield
\begin{equation}\label{eq:double}
\begin{split}
 \|\la p\ra W(s)\|_\infty
 &\leq C_0(1+s)+C\int_0^s(s-\tau)\alpha(\tau)\|\la p\ra W(\tau)\|_\infty\dd\tau\\
 &\leq2C_0+C\int_0^s(1-\tau)\alpha(\tau)\|\la p\ra W(\tau)\|_\infty\dd\tau.
\end{split}
\end{equation}
The last kernel is integrable. Gr\"onwall's inequality bounds $\|\la p\ra W(s)\|_\infty$ uniformly and substitute into the estimate for $Z$ then gives
\begin{equation*}
    \begin{split}
        \|\la p\ra Z\|_{L^{\infty}}&\les 1+\beta(s),\\
  \|\la p\ra \pr_p\sigma\|_{L^{\oo}}=      \|\la p\ra W\|_{L^{\infty}}&\les 1.
    \end{split}
\end{equation*}
Hence,
\begin{equation*}
    \begin{split}
        \|\la p \ra\nab_q\sigma\|_{L^{\infty}}\leq \|\la p\ra Z\|_{L^{\infty}}+\|\pr_q\Gamma\|_{L^{\infty}}\|\la p\ra\pr_p\sigma\|_{L^{\infty}}\les 1+\beta(s).
    \end{split}
\end{equation*}
Therefore,
\begin{equation}\label{eq:d^2E}
    \begin{split}
       | \pr_q^2 E(s,q)|=2\left|\int_{\R} \sigma\, \pr_q\sigma dp\right|\les \|\pr_q\sigma\|_{L^{\infty}}\|\la p\ra^2\sigma\|_{L^{\infty}}\les \|\pr_q\sigma\|_{L^{\infty}}.
    \end{split}
\end{equation}
To improve the unweighted bound for $\pr_q\sigma$ and estimate $\pr_q^2\Gamma$, differentiate \eqref{eq:dBurg} once more:
\begin{equation*}
    \begin{split}
        \pr_s (\pr^2_q\Gamma)+\Gamma\pr_q(\pr_q^2\Gamma)+3\pr_q\Gamma \, \pr_q^2\Gamma=\lambda \alpha(s)\pr_q^2 E(s,q).
    \end{split}
\end{equation*}
Using Lemma \ref{lem:prop} with $u=\pr_q^2\Gamma, b=\Gamma, c=3\pr_q\Gamma$ and $g=\lambda\alpha(s)\pr_q^2E(s,q)$, we have
\begin{equation*}
    \begin{split}
        \|\pr_q^2\Gamma(s)\|_{\oo}
        &\leq C\int_0^s\alpha(\tau)\|\pr_q\sigma(\tau)\|_{L^{\oo}}\dd\tau,\\
        \|\pr_q\sigma(s)\|_{L^{\oo}}        &\leq C_0+C\int_0^s\|\pr_q^2\Gamma(\tau)\|_{L^{\oo}}\dd\tau.
    \end{split}
\end{equation*}
For the second inequality, apply  Lemma \ref{lem:Hamilton} to the first equation in \eqref{eq:D1}, using $\pr_q\Gamma\geq0$ and the bound for $p\pr_p\sigma$. Substituting the first inequality into the second gives
\[
 \|\pr_q\sigma(s)\|_\infty \leq C_0+C\int_0^s(1-\tau)\alpha(\tau)\|\pr_q\sigma(\tau)\|_\infty\dd\tau.
\]
The same  Gr\"onwall argument as in \eqref{eq:double}
therefore proves
\begin{equation*}
    \begin{split}
\|\pr_q^2\Gamma(s)\|_{L^{\infty}}\les\beta(s),\quad \|\nab_q\sigma(s)\|_{L^{\infty}}\les1.
    \end{split}
\end{equation*}
The estimate for $\pr_q^2E$ now follows from \eqref{eq:d^2E}.
\end{proof}

\section{Asymptotics for the electric field}\label{sec:electric-field}
In this section, we derive the asymptotics of the electric fields.
Define
\begin{equation}\label{eq:Bdef}
    B(s,q)=\Gamma(s,q)-\lambda\beta(s)E(s,q).
\end{equation}
We  first prove the uniform bounds for $B$ and its first spatial derivative in the following lemma, which will be used to obtain the asymptotics of the electric field.
\begin{lemma}
    Under \eqref{eq:data}, we have for $0\le s<1$
    \begin{equation}\label{eq:B}
        \begin{split}
\|B(s)\|_{\oo}+\|\pr_qB(s)\|_{\oo}\les 1.
        \end{split}
    \end{equation}
\end{lemma}
\begin{proof}

The equation for $\sigma^2$ gives
\begin{equation}\label{eq:Es}
    \pr_s E(s,q)=-\int_{\R}p\sigma^2(s,q,p)\dd p-\Gamma(s,q)\pr_qE(s,q).
\end{equation}
Indeed, expanding \eqref{eq:sigma-K} with $\sigma^2$ gives
\begin{equation}
\pr_s(\sigma^2)+\pr_q((p+\Gamma)\sigma^2)-\pr_p(p\pr_q\Gamma\sigma^2)=0.
\end{equation}
Integrating over $p\in \R$ gives the continuity equation
\[
\pr_s\rho+\pr_q\left(\int_{\R}p\sigma^2\dd p+\Gamma\rho\right)=0,
\]
and differentiating the formula for $E$ yields \eqref{eq:Es}. Proposition \ref{prop:der} implies
\begin{equation}\label{eq:pints}
 \left\|\int_{\R}p\sigma^2(s,q,p)\dd p\right\|_{L^\infty_q}+\left\|2\int_{\R}p\sigma(s,q,p)\pr_q\sigma(s,q,p)\dd p\right\|_{L^\infty_q}
 \le C.
\end{equation}
For the second bound, for example, use
$|\sigma(s,q,p)|\le C_0\la p\ra^{-3}$ and $\|\pr_q\sigma(s)\|_{L^\infty}\le C$, so that $|p|\la p\ra^{-3}$ is integrable. Combining \eqref{eq:Burg} and \eqref{eq:Es}, and using $\beta'=\alpha$, gives
\begin{equation}\label{eq:BPDE}
 \begin{split}
  (\pr_s+\Gamma\pr_q)B
     &=\lambda\beta(s)\int_{\R}p\sigma^2(s,q,p)\dd p,\\
  (\pr_s+\Gamma\pr_q)(\pr_qB) +(\pr_q\Gamma)\pr_qB
&=2\lambda\beta(s)\int_{\R}p\sigma(s,q,p)\pr_q\sigma(s,q,p)\dd p.
 \end{split}
\end{equation}
Both $B$ and $\pr_qB$ vanish at $s=0$. Since $\pr_q\Gamma\ge0$, Lemma \ref{lem:prop} and \eqref{eq:pints} show that
\begin{equation}\label{eq:Bbd}
\|B(s)\|_{L^\infty}+\|\pr_qB(s)\|_{L^\infty} \le C\int_0^s\beta(u)\dd u\le C, \qquad 0\le s<1.
\end{equation}

\end{proof}

\subsection{Convergence of electric field} We have the following asymptotic expansion for the electric field.
\begin{lemma}
    Under \eqref{eq:data}, there exists $E_{\infty}\in W^{2,\oo}$ such that
\begin{equation}\label{eq:asymptotic-E}
    \begin{split}
    E(s)&=E_{\oo}+\lambda\sqrt{1-s^2}E_{\oo}\pr_q E_{\oo}+\cO(1-s),\\   \pr_q E(s)  &=\pr_q E_{\oo}+\cO(\sqrt{1-s^2}).
    \end{split}
\end{equation}
    
\end{lemma}

\begin{proof}

Differentiating \eqref{eq:Es} gives
\begin{equation}\label{eq:Eqs}
 \pr_s\pr_qE=-2\int_{\R}p\sigma\pr_q\sigma\dd p  -(\pr_q\Gamma)\pr_qE-\Gamma\pr_q^2E.
\end{equation}
\eqref{eq:pints}, Proposition \ref{prop:der} and Proposition \ref{prop:mom} therefore imply
\begin{equation}\label{eq:Etbd}
 \|\pr_sE(s)\|_{L^\infty}+\|\pr_s\pr_qE(s)\|_{L^\infty}\le C(1+\beta(s)).
\end{equation}
Since $\beta\in L^1(0,1)$, we can define
\begin{equation}\label{eq:Einf}
 E_\infty(q)=E(0,q)-\int_0^1\left[ \int_{\R}p\sigma^2(s,q,p)\dd p+\Gamma(s,q)\pr_qE(s,q)\right]\dd s.
\end{equation}
The integral converges uniformly in $q$. Integration of \eqref{eq:Eqs} likewise yields a uniform limit for $\pr_qE(s)$; this limit is the derivative of $E_\infty$. In particular, $E_\infty\in C^1(\R)$, its value and derivative are bounded,
and
\begin{equation}\label{eq:Econv}
 \|E(s)-E_\infty\|_{L^\infty} +\|\pr_qE(s)-\pr_qE_\infty\|_{L^\infty} \le C\sqrt{1-s^2}.
\end{equation}
We used
$\int_s^1\beta(u)\dd u=\sqrt{1-s^2}$ and $1-s\le\sqrt{1-s^2}$. Moreover, $\|E_\infty\|_{L^\infty}\le M/2$, and the uniform bound for $\pr_q^2E(s)$ implies that $\pr_qE_\infty$ is globally Lipschitz.

For the modified scattering argument, we need one further term in this expansion. Substitute $\Gamma=B+\lambda\beta E$ into
\eqref{eq:Es} and integrate from $s$ to $1$ to obtain
\begin{equation}\label{eq:Eint}
 \begin{split}
 E(s,q)-E_\infty(q)  ={}&\int_s^1\left[\int_{\R}p\sigma^2(\tau,q,p)\dd p+B(\tau,q)\pr_qE(\tau,q)\right]\dd \tau\\
&+\lambda\int_s^1\beta(\tau)E(\tau,q)\pr_qE(\tau,q)\dd \tau.
 \end{split}
\end{equation}
The first integral is bounded by $C(1-s)$, uniformly in $q$. By \eqref{eq:Econv}, replacing
$E(\tau,q)\pr_qE(\tau,q)$ by $E_\infty(q)\pr_qE_\infty(q)$ in the second integral introduces an error bounded by
\[
 C\int_s^1\beta(\tau)\sqrt{1-\tau^2}\dd \tau=\frac C2(1-s^2)\le C(1-s).
\]
Consequently,
\begin{equation}\label{eq:Eexp}
 \left\|E(s)-E_\infty -\lambda\sqrt{1-s^2}E_\infty\pr_qE_\infty\right\|_{L^\infty} \le C(1-s), 
\end{equation}
where we used $$
\int_s^1 \frac{\tau}{\sqrt{1-\tau^2}}\,d\tau=\sqrt{1-s^2}. $$
\end{proof}

\section{Modified scattering}\label{sec:main}

In this section we prove the modified scattering.

\begin{proof}[Proof of Theorem \ref{thm:main}]
Proposition \ref{prop:time} gives global existence and uniqueness. It remains to prove the scattering assertion. Define
\begin{equation*}
    \begin{split}
\nu(s,q,p):=\ga(s,\Phi^1(s,q,p),\Phi^2(s,q,p))
    \end{split}
\end{equation*}
where
\begin{equation*}
    \begin{split}
        \Phi^1(s,q,p)&=q+(s-1)p+\lambda(s-1)\beta(s)E_{\infty}(q),\\
        \Phi^2(s,q,p) &=p+\lambda\beta(s)E_{\infty}(q).
    \end{split}
\end{equation*}
A direct computation gives
\begin{equation*}
    \begin{split}
      \pr_s\nu&=-\lambda(1-s)\alpha(s)E_{\oo}(q)\pr_q\ga (s,\Phi) \\
        &\quad +\lambda\alpha(s)[E_{\infty}(q)-E(s,\Phi^1)]\pr_p\ga (s,\Phi).
    \end{split}
\end{equation*}
We rewrite this as
\begin{equation}\label{eq:nus}
    \begin{split}
        \pr_s\nu&=-\lambda(1-s)\alpha(s)E_{\oo}(q)[\pr_q\ga+\lambda\beta(s)\pr_q E\cdot\pr_p\ga](s,\Phi)\\
        &\quad +\lambda\alpha(s)[E_{\infty}(q)-E(s,\Phi^1)+\lambda(1-s)\beta(s)E_{\oo}(q)\pr_q E(s,\Phi^1)](\pr_p\ga)(s,\Phi).
    \end{split}
\end{equation}
The chain rule for $\sigma$ gives, at any point $(s,Q,P)$,
\begin{equation}\label{eq:nuder}
    \begin{split}
        [\pr_q\ga+\lambda\beta(s)(\pr_qE)\pr_p\ga](s,Q,P)&=[\pr_q\sigma-(\pr_qB)\pr_p\sigma](s,Q,P-\Gamma(s,Q)), \\
        \pr_p\ga(s,Q,P)&=(\pr_p\sigma)(s,Q,P-\Gamma(s,Q)).
    \end{split}
\end{equation}
The first expression is uniformly bounded, by \eqref{eq:Bbd} and the bounds for $\pr_q\sigma$ and $\pr_p\sigma$. Thus the first term on the right-hand side of \eqref{eq:nus} is bounded by
$C(1-s)\alpha(s)$.

For the second term, we claim that
\begin{equation}\label{eq:nufld}
 \begin{split}
 &\left|E_\infty(q)-E(s,\Phi^1(s,q,p))
       +\lambda(1-s)\beta(s)E_\infty(q)\pr_qE(s,\Phi^1(s,q,p))\right|\\
 &\hspace{4cm}\le C(1-s)\la p\ra.
 \end{split}
\end{equation}
First take $p=0$. Since
\[
 \Phi^1(s,q,0)-q=-\lambda(1-s)\beta(s)E_\infty(q),
\]
its absolute value is at most $C\sqrt{1-s^2}$. Taylor expansion in
\eqref{eq:Eexp}, using the Lipschitz continuity of $\pr_qE_\infty$, gives, uniformly in $q$,
\begin{equation*}
 \begin{split}
 E(s,\Phi^1(s,q,0))
   &=E_\infty(q)+\lambda\big[\sqrt{1-s^2}-(1-s)\beta(s)\big] E_\infty(q)\pr_qE_\infty(q)+\cO(1-s),\\ \pr_qE(s,\Phi^1(s,q,0))
   &=\pr_qE_\infty(q)+\cO(\sqrt{1-s^2}).
 \end{split}
\end{equation*}
It follows that the expression inside the absolute value in
\eqref{eq:nufld}, at $p=0$, equals
\[
 \lambda\big[2(1-s)\beta(s)-\sqrt{1-s^2}\big] E_\infty(q)\pr_qE_\infty(q)+\cO(1-s).
\]
Since
\begin{equation}\label{eq:cancel}
 2(1-s)\beta(s)-\sqrt{1-s^2}
       =-\frac{(1-s)^2}{\sqrt{1-s^2}},
\end{equation}
this proves \eqref{eq:nufld} when $p=0$. For arbitrary $p$, $\Phi^1(s,q,p)-\Phi^1(s,q,0)=-(1-s)p$. The uniform bounds for $\pr_qE$ and $\pr_q^2E$, together with the boundedness of $(1-s)\beta(s)$, show that replacing $0$ by $p$ changes the expression by at most $C(1-s)|p|$. This proves
\eqref{eq:nufld} for every $p$.

To use the velocity weight in the estimate for $\pr_p\sigma$, we compare the two velocity arguments. Equations \eqref{eq:Econv} and
\eqref{eq:Bbd} imply
\begin{equation*}
 \begin{split}
 |\Phi^2(s,q,p)-\Gamma(s,\Phi^1(s,q,p))-p|
 &\le\lambda\beta(s)|E_\infty(q)-E(s,\Phi^1(s,q,p))|     +\|B(s)\|_{L^\infty}\\
 &\le C+C(1-s)\beta(s)|p|.
 \end{split}
\end{equation*}
Indeed, the field difference is bounded by
$C\sqrt{1-s^2}+C(1-s)|p|$, and $\beta(s)\sqrt{1-s^2}=s$.
Since $(1-s)\beta(s)\to0$, for $s$ sufficiently close to $1$ we obtain
\begin{equation}\label{eq:nuv}
 C^{-1}\la p\ra \le\la\Phi^2(s,q,p)-\Gamma(s,\Phi^1(s,q,p))\ra \le C\la p\ra.
\end{equation}
For $s$ away from $1$, \eqref{eq:nuv} holds directly, since $|\Phi^2-\Gamma(s,\Phi^1)-p|\le|\lambda|\beta(s)\|E_\infty\|_{L^\infty} +\|\Gamma(s)\|_{L^\infty}\les\beta(s)$.
Hence \eqref{eq:nuv} holds for all $0\le s<1$.

The second identity in \eqref{eq:nuder} and the bound for $\la p\ra\pr_p\sigma$ now give
\[
 |(\pr_p\gamma)(s,\Phi(s,q,p))|\le C\la p\ra^{-1}.
\]
Combining this estimate with \eqref{eq:nufld}, every term in \eqref{eq:nus} is bounded by an integrable function:
\begin{equation}\label{eq:nuL1}
 \|\pr_s\nu(s)\|_{L^\infty}\le C(1-s)\alpha(s), \qquad \int_0^1(1-s)\alpha(s)\dd s=1.
\end{equation}
Thus the uniform limit
\begin{equation}\label{eq:muinf}
\mu_\infty(q,p):=\lim_{s\uparrow1}\nu(s,q,p)
\end{equation}
exists, and
\begin{equation}\label{eq:nuconv}
 \|\nu(s)-\mu_\infty\|_{L^\infty} \le C\int_s^1(1-u)\alpha(u)\dd u =C\sqrt{\frac{1-s}{1+s}}.
\end{equation}
In particular $\mu_\infty\in L^\infty$. For each $s<1$, the map $\Phi(s,\cdot)$ is a global $C^1$ diffeomorphism with determinant one: it is the composition of the velocity translation $(q,p)\mapsto(q,p+\lambda\beta(s)E_\infty(q))$ and the spatial shear $(Q,P)\mapsto(Q+(s-1)P,P)$.
Consequently,
\[
 \iint_{\R^2}|\nu(s,q,p)|^2\dd q\dd p   =\iint_{\R^2}|\gamma(s,q,p)|^2\dd q\dd p=M.
\]
Fatou's lemma gives $\|\mu_\infty\|_{L^2}^2\le M$, so that $\mu_\infty\in L^2\cap L^\infty$.

Next, we  express \eqref{eq:nuconv} in the original variables. By \eqref{eq:lens} and $\beta(\tanh t)=\sinh t$, direct substitution yields
\begin{equation*}
 \begin{split}
 \nu(\tanh t,q,p)  =\mu\Big(&t,\ q\cosh t-e^{-t}p         -\frac{\lambda}{2}(1-e^{-2t})E_\infty(q),\\
&q\sinh t+e^{-t}p+\frac{\lambda}{2}(1-e^{-2t})E_\infty(q)\Big)\\
=\mu\big(&t,\mathcal{X}(t,q-p,p),\mathcal{V}(t,q-p,p)\big).
 \end{split}
\end{equation*}
Since $\sqrt{(1-\tanh t)/(1+\tanh t)}=e^{-t}$, renaming $(q,p)$ as $(x,v)$ in \eqref{eq:nuconv} proves
\[
 \left\|\mu\big(t,\mathcal{X}(t,x-v,v),\mathcal{V}(t,x-v,v)\big) -\mu_\infty(x,v)\right\|_{L^\infty_{x,v}}\le Ce^{-t}.
\]

We finally identify the relation between $E_\infty$ and $\mu_\infty$. The propagation of the velocity arguments, together with the cubic
velocity bound for $\sigma$, gives
\[
 \sup_{0\le s<1}\|\la p\ra^3\nu(s)\|_{L^\infty} +\|\la p\ra^3\mu_\infty\|_{L^\infty}\le C.
\]
Consequently, the uniform convergence of $\nu(s)$ to $\mu_\infty$ and dominated convergence imply that, for every
$\varphi\in C_c^\infty(\R)$,
\[
\lim_{s\uparrow1}\iint_{\R^2}\varphi(q)\nu^2(s,q,p)\dd q\dd p=\iint_{\R^2}\varphi(q)\mu_\infty^2(q,p)\dd q\dd p.
\]
Since $\det D\Phi=1$ and $\Phi^1(s,q,p)+(1-s)\Phi^2(s,q,p)=q$, changes of variables yield
\[
 \begin{split}
\iint_{\R^2}\varphi(q)\nu^2(s,q,p)\dd q\dd p &=\iint_{\R^2}\varphi(q+(1-s)p)\gamma^2(s,q,p)\dd q\dd p\\
 &=\iint_{\R^2} \varphi\big(q+(1-s)(p+\Gamma(s,q))\big) \sigma^2(s,q,p)\dd q\dd p.
 \end{split}
\]
Using $\pr_qE=\int_\R\sigma^2\dd p$, we therefore obtain
\[
 \begin{split}
&\left|\iint_{\R^2}\varphi(q)\nu^2(s,q,p)\dd q\dd p-\int_\R\varphi(q)\pr_qE(s,q)\dd q\right|\\
&\qquad\le(1-s)\|\varphi'\|_{L^\infty}
 \left(
\|p\sigma(s)\|_{L^2}\|\sigma(s)\|_{L^2}+M\|\Gamma(s)\|_{L^\infty}\right)\\
 &\qquad\le C\|\varphi'\|_{L^\infty} (1-s)(1+\beta(s)) \longrightarrow0.
 \end{split}
\]
Passing to the limit and using the uniform convergence of
$\pr_qE(s)$ to $\pr_qE_\infty$, we conclude that
\[
 \int_\R\varphi(q)\pr_qE_\infty(q)\dd q=\iint_{\R^2}\varphi(q)\mu_\infty^2(q,p)\dd q\dd p.
\]
The continuity of $\mu_\infty$ and its cubic velocity bound imply that $\int_\R\mu_\infty^2(q,p)\dd p$ is continuous in $q$. Hence
\[
\pr_qE_\infty(q)=\int_\R\mu_\infty^2(q,p)\dd p \qquad\text{for every }q\in\R.
\]

Moreover, $E(s,\pm\infty)=\pm M/2$ for every $s<1$. The uniform convergence $E(s)\to E_\infty$ preserves these limits, so $E_\infty(\pm\infty)=\pm M/2$. Integrating the preceding identity
gives
\[
 \iint_{\R^2}\mu_\infty^2(q,p)\dd q\dd p=E_\infty(+\infty)-E_\infty(-\infty)=M.
\]
It follows that 
\[
{
 E_\infty(q) =\frac12\iint_{\R^2} \sgn(q-y)\mu_\infty^2(y,p)\dd y\dd p =E[\mu_\infty](q).
 }
\]
This completes the proof of Theorem \ref{thm:main}.

\end{proof}

\section{Modified wave operators}\label{sec:wave}
In this section we prove Theorem \ref{thm:main2}, which requires solving \eqref{eq:lensP} backward from $s=1$. Neither unknown used so far is suitable for this. The equation for $\gamma$ contains the nonintegrable weight $\alpha(s)$, while $\sigma$ is defined through $\Gamma$, which is obtained by integrating \eqref{eq:Burg} forward from $s=0$ and is not available at $s=1$. We instead shift the momentum by the explicit leading part of $\Gamma$, which only involves the field at the same time:
\begin{equation}\label{eq:wave:shear}
\widetilde\gamma(s,q,p)=\gamma\big(s,q,p+\lambda\beta(s)E(s,q)\big)=\sigma\big(s,q,p-B(s,q)\big),\qquad B=\Gamma-\lambda\beta E.
\end{equation}
The equation for $\tga$ has coefficients of size $\beta(s)$, which is integrable. The price is a loss of one power of $p$ in the derivative estimates, caused by the term $\lambda\beta\,p\,\pr_q^2E\,\pr_p\tga$ in the equation for $\pr_q\tga$. We recover it with a vector field $Z=A\pr_q+C\pr_p$ commuting with the equation. By the continuity equation, the second derivative of $A$ along the flow equals $\lambda\alpha(s)\pr_qE\,A$, see \eqref{eq:AC-cancellation}, so after two time integrations only the kernel $(1-s)\alpha(s)$ appears, as in Section~\ref{sec:correction}.

For convenience, write the density and corrected current as
\[\rho(s,q)=\int_\R\widetilde\gamma(s,q,p)^2\dd p,  \qquad j(s,q)=\int_\R p\widetilde\gamma(s,q,p)^2\dd p.
\]
A direct computation gives that $\widetilde{\ga}$ solves
\begin{equation}\label{eq:tK}
    \begin{split}
     \pr_s\widetilde{\ga}+\{\tga,\tK \}=0,\quad \tK(s,q,p)=\frac{p^2}{2}+\lambda\beta(s)\of{pE(s,q)-\int_0^q j(s,y)dy}   
    \end{split}
\end{equation}
To propagate the moments and derivatives, we need to compute the derivatives of $\tK$:
\begin{equation}
    \begin{split}
  & \pr_q\tK=     \lambda \beta(s)(p\pr_q E-j(s,q)),\quad \pr_p\tK=p+\lambda\beta(s)E(s,q),\\
& \pr_q^2\tK=\lambda\beta(s)(p\pr_q^2E-\pr_qj),\quad \pr_p\pr_q\tK=\lambda\beta(s)\pr_q E,\quad \pr_p^2 \tK=1,\\
&\pr_q^3\tK=\lambda\beta(s)(p\pr_q^3E-\pr_q^2j),\quad \pr_q^2\pr_p\tK=\lambda\beta(s)\pr_q^2E,\quad \pr_q\pr_{p}^2\tK=\pr_p^3\tK=0.
    \end{split}
\end{equation}

\begin{theorem}\label{thm:3}
Let the prescribed final datum $\ga_\oo$  satisfy
\begin{equation}\label{eq:tga-initial}
\sum_{j+k\le2}\sum_{r\in\{2,\oo\}}\|\la p\ra^3\pr_q^j\pr_p^k\ga_\oo\|_r\le c_0<\oo.
\end{equation}
There exists $s_*\in[1/2,1)$, depending on $c_0$ and $\lambda$, such that \eqref{eq:tK} has a solution
\[
\tga\in C([s_*,1];L^2\cap L^\oo)\cap C^1([s_*,1)\times\R^2),\qquad \tga(1,q,p)=\ga_\oo(q,p),
\]
with
\begin{equation}\label{eq:tga-solution-bound}
\sup_{s_*\le s\le1}\sum_{l=0}^2\sum_{j+k=l}\sum_{r\in\{2,\oo\}}\|\la p\ra^{3-l}\pr_q^j\pr_p^k\tga(s)\|_r\le4c_0.
\end{equation}
 The first derivatives of $\tga$ in $(q,p)$ extend continuously to $s=1$. The solution is unique. Furthermore,
\begin{equation}\label{eq:tga-endpoint-rate}
\|\tga(s)-\ga_\oo\|_2+\|\tga(s)-\ga_\oo\|_\oo\les\sqrt{1-s^2},\qquad s\rightarrow1.
\end{equation}
\end{theorem}

The solution will be constructed by Picard iteration. We first establish a priori estimates for the nonlinear equation and then verify that these estimates also hold uniformly for the linear iteration. In this section, $K$ denotes a  constant, whereas implicit constants may depend on $c_0$ and $\lambda$. We choose $s_*$ sufficiently close to $1$ that
\begin{equation}\label{eq:tga-shortness}
K(1+|\lambda|c_0^2)\sqrt{1-s_*^2}\le\frac1{100},
\end{equation}
where the  constant $K$ is fixed sufficiently large to accommodate the estimates below.

For an arbitrary $s_1\in[s_*,1)$, we make the bootstrap assumption
\begin{equation}\label{eq:bootstrap}
\sup_{s_1\le s\le1}\sum_{l=0}^2\sum_{j+k=l}\sum_{r\in\{2,\oo\}}\|\la p\ra^{3-l}\pr_q^j\pr_p^k\tga(s)\|_r\le4c_0.
\end{equation}
The definitions of $E$ and $j$ give
\begin{equation}\label{eq:bound-Ej}
\begin{aligned}
\|E\|_\oo
&\le\frac12\|\tga\|_2^2,\\
\|j\|_\oo+\|\pr_qE\|_\oo
&\le K\|\la p\ra^3\tga\|_\oo^2,\\
\|\pr_qj\|_\oo+\|\pr_q^2E\|_\oo
&\le K\|\la p\ra^3\tga\|_\oo
          \|\la p\ra^2\pr_q\tga\|_\oo,\\
\|\pr_q^2j\|_\oo+\|\pr_q^3E\|_\oo
&\le K\left(
\|\la p\ra^2\pr_q\tga\|_\oo^2+\|\la p\ra^3\tga\|_\oo \|\la p\ra\pr_q^2\tga\|_\oo\right).
\end{aligned}
\end{equation}
Indeed,
\[
\begin{aligned}
\pr_qj&=2\int_\R p\tga\pr_q\tga\dd p,
&\pr_q^2E&=2\int_\R\tga\pr_q\tga\dd p,\\
\pr_q^2j&=2\int_\R p\bigl((\pr_q\tga)^2+\tga\pr_q^2\tga\bigr)\dd p,
&\pr_q^3E&=2\int_\R\bigl((\pr_q\tga)^2+\tga\pr_q^2\tga\bigr)\dd p.
\end{aligned}
\]
The products in the second line are bounded by $Kc_0^2\la p\ra^{-4}$, and $\int_\R(1+|p|)\la p\ra^{-4}\dd p<\oo$. Consequently,
\begin{equation}\label{eq:tga-field-bootstrap}
\sup_{s_1\le s\le1}\left[\|E(s)\|_\oo+
\sum_{k=0}^2\left(\|\pr_q^kj(s)\|_\oo+\|\pr_q^{k+1}E(s)\|_\oo\right)\right]\le Kc_0^2.
\end{equation}

\subsection{A priori estimates}

\begin{lemma}\label{lem:tga-bootstrap}
Suppose that a smooth solution of \eqref{eq:tK}, with final datum satisfying \eqref{eq:tga-initial}, obeys \eqref{eq:bootstrap} on $[s_1,1]$. If \eqref{eq:tga-shortness} holds, then
\begin{equation}\label{eq:tga-bootstrap-improved}
\sup_{s_1\le s\le1}\sum_{l=0}^2\sum_{j+k=l}\sum_{r\in\{2,\oo\}}\|\la p\ra^{3-l}\pr_q^j\pr_p^k\tga(s)\|_r\le2c_0.
\end{equation}
The choice of $s_*$ is independent of $s_1$.
\end{lemma}

\begin{proof}
Write
\[
\D=\pr_s+\{\cdot,\tK\}=\pr_s+(p+\lambda\beta E)\pr_q+\lambda\beta(j-p\pr_qE)\pr_p.
\]
Throughout the proof, $s_*\le s_1\le s<1$. The field of $\D$ is divergence free. We keep track of the length of $[s_1,1]$; this is what improves $4c_0$ to $2c_0$. Transport estimates involving the final time $s=1$ are first applied on compact time intervals with upper endpoint $s<1$. We then let $s\rightarrow 1$, using the prescribed bounds and integrability of $\beta(s)$.

\smallskip
\noindent\textbf{Step 1. No derivatives and three velocity moments: $l=0$.}
First,
\begin{equation}\label{eq:zero}
\|E\|_\oo\le\frac12\|\tga\|_2^2,\qquad
\|j\|_\oo+\|\pr_qE\|_\oo\le K\|\la p\ra^2\tga\|_\oo^2\le Kc_0^2.
\end{equation}
If $\D v=0$, then, for $0\le m\le3$,
\[
\D(\la p\ra^mv)=m\lambda\beta\frac{p(j-p\pr_qE)}{1+p^2}\la p\ra^mv.
\]
Lemma \ref{lem:Hamilton} and Gronwall's inequality therefore give, in the backward time direction,
\begin{equation}\label{eq:tga-weighted-transport}
\begin{aligned}
\|\la p\ra^mv(s_1)\|_r&\le\exp\left(K|\lambda|c_0^2\int_{s_1}^s\beta(u)\dd u\right) \|\la p\ra^mv(s)\|_r\\
&\le\left(1+K|\lambda|c_0^2\sqrt{1-s_*^2}\right) \|\la p\ra^mv(s)\|_r,
\qquad r\in\{2,\oo\}.
\end{aligned}
\end{equation}
We used \eqref{eq:tga-shortness} and $\int_{s_1}^1\beta(u)\dd u=\sqrt{1-s_1^2}$. In particular,
\[
\D(\la p\ra^3\tga)=-3\lambda\beta\,p\la p\ra(p\pr_qE-j)\tga,
\]
and taking $v=\tga$, $m=3$ in \eqref{eq:tga-weighted-transport} gives
\begin{equation}\label{eq:tga-zero-bound}
\|\la p\ra^3\tga(s_1)\|_r\le\left(1+K|\lambda|c_0^2\sqrt{1-s_*^2}\right)\|\la p\ra^3\ga_\oo\|_r.
\end{equation}

\smallskip
\noindent\textbf{Step 2. First derivatives and two velocity moments: $l=1$.} Direct differentiation yields
\[
\begin{aligned}
\D(\pr_q\tga)
&=-\lambda\beta(\pr_qE)\pr_q\tga +\lambda\beta(p\pr_q^2E-\pr_qj)\pr_p\tga,\\
\D(\pr_p\tga)&=-\pr_q\tga+\lambda\beta(\pr_qE)\pr_p\tga.
\end{aligned}
\]
The factor $p\pr_q^2E$ leads to a loss of velocity moments in a direct estimate. To use the cancellation supplied by the continuity equation, let $A=A(s,q,p)$ and $C=C(s,q,p)$ solve
\begin{equation}\label{eq:AC}
\begin{aligned}
\D A&=C+\lambda\beta(\pr_qE)A,\\
\D C&=\lambda\beta\bigl[(\pr_qj-p\pr_q^2E)A-(\pr_qE)C\bigr],
\end{aligned}
\end{equation}
with
\[
A(s_1,q,p)=a,\qquad C(s_1,q,p)=c,\qquad (a,c)\in\{(1,0),(0,1),(1,1)\}.
\]
The dependence of $A,C$ on $s_1,a,c$ is suppressed. Set $Z=A\pr_q+C\pr_p$. Then
\[
\begin{aligned}
[\D,Z]
={}&\bigl(\D A-\lambda\beta(\pr_qE)A-C\bigr)\pr_q\\
&+\bigl(\D C+\lambda\beta(p\pr_q^2E-\pr_qj)A+\lambda\beta(\pr_qE)C\bigr)\pr_p=0.
\end{aligned}
\]
Consequently,
\begin{equation}\label{eq:tga-commuting}
\D(Z\tga)=0,\qquad \D(Z^2\tga)=0.
\end{equation}

We first estimate $A,C$ and their differences from $a,c$. Multiplying the equation for $\tga$ by $2\tga$ and integrating in $p$ gives
\begin{equation}\label{eq:tga-continuity}
\pr_s\rho+\pr_q(j+\lambda\beta E\rho)=0,\qquad \pr_sE=-j-\lambda\beta E\rho,\qquad \rho=\pr_qE.
\end{equation}
In particular,
\[
\D(\pr_qE)=p\pr_q^2E-\pr_qj-\lambda\beta(\pr_qE)^2.
\]
Using $\beta'=\alpha$, we obtain
\begin{equation}\label{eq:AC-cancellation}
\begin{aligned}
\D^2A={}&\D C+\lambda\alpha(\pr_qE)A+\lambda\beta\D(\pr_qE)A+\lambda\beta(\pr_qE)\D A\\
={}&\lambda\beta\bigl[(\pr_qj-p\pr_q^2E)A-(\pr_qE)C\bigr] +\lambda\alpha(\pr_qE)A\\
&+\lambda\beta\bigl[p\pr_q^2E-\pr_qj-\lambda\beta(\pr_qE)^2\bigr]A+\lambda\beta(\pr_qE)\bigl[C+\lambda\beta(\pr_qE)A\bigr]\\
={}&\lambda\alpha(\pr_qE)A.
\end{aligned}
\end{equation}
The system \eqref{eq:AC} can be constructed on compact time intervals by solving this second-order linear transport equation with
\[
A(s_1)=a,\qquad \D A(s_1)=c+\lambda\beta(s_1)\pr_qE(s_1)a,
\]
and setting $C=\D A-\lambda\beta(\pr_qE)A$.

Applying Lemma~\ref{lem:Hamilton} successively to $A$ and $\D A$, and using Fubini's theorem, gives
\begin{equation}\label{eq:AC-double-transport}
\|A(s)\|_\oo\le\|A(s_1)\|_\oo+(s-s_1)\|\D A(s_1)\|_\oo+\int_{s_1}^s(s-u)\|\D^2A(u)\|_\oo\dd u.
\end{equation}
Consequently,
\begin{equation}\label{eq:AC-A-integral}
\begin{aligned}
\|A(s)\|_\oo\le{}&|a|+(s-s_1)|c|+K|\lambda|c_0^2|a|(s-s_1)\beta(s_1)\\
&+K|\lambda|c_0^2\int_{s_1}^s(s-u)\alpha(u)\|A(u)\|_\oo\dd u.
\end{aligned}
\end{equation}
We use the identities
\begin{equation}\label{eq:AC-kernels}
\begin{aligned}
(s-s_1)\beta(s_1)
&\le(1-s_1)\beta(s_1)=s_1\sqrt{\frac{1-s_1}{1+s_1}}\le\sqrt{1-s_*^2},\\
\int_{s_1}^s(s-u)\alpha(u)\dd u&\le\int_{s_1}^1(1-u)\alpha(u)\dd u=\sqrt{\frac{1-s_1}{1+s_1}}\le\sqrt{1-s_*^2}.
\end{aligned}
\end{equation}
Taking the supremum in \eqref{eq:AC-A-integral} and using \eqref{eq:tga-shortness}, we absorb the last term and obtain
\[
\sup_{s_1\le s<1}\|A(s)\|_\oo\le3.
\]
Applying the same calculation to $A-a$, whose initial value is zero, gives the sharper estimate
\begin{equation}\label{eq:AC-A-close}
\sup_{s_1\le s<1}\|A(s)-a\|_\oo\le K\bigl[(1-s_*)+|\lambda|c_0^2\sqrt{1-s_*^2}\bigr]\le K(1+|\lambda|c_0^2)\sqrt{1-s_*^2}.
\end{equation}
For the second coefficient,
\[
\D\log\la p\ra=\lambda\beta\frac{p(j-p\pr_qE)}{1+p^2},\qquad\|\D\log\la p\ra\|_\oo\le K|\lambda|c_0^2\beta.
\]
Subtracting the constant $c$ in \eqref{eq:AC} and dividing by $\la p\ra$ yields
\begin{equation}\label{eq:AC-C-equation}
\begin{aligned}
\D\left(\frac{C-c}{\la p\ra}\right)+\left(\lambda\beta\pr_qE+\D\log\la p\ra\right)\frac{C-c}{\la p\ra}=\lambda\beta\frac{(\pr_qj-p\pr_q^2E)A-c\pr_qE}{\la p\ra}.
\end{aligned}
\end{equation}
The initial value is zero. By \eqref{eq:tga-field-bootstrap}, the coefficient and source are bounded by $K|\lambda|c_0^2\beta$. Gronwall's inequality therefore gives
\begin{equation}\label{eq:AC-C-close}
\begin{aligned}
\sup_{s_1\le s<1}\left\|\frac{C(s)-c}{\la p\ra}\right\|_\oo
&\le K|\lambda|c_0^2\sqrt{1-s_*^2}\exp\left(K|\lambda|c_0^2\sqrt{1-s_*^2}\right)\\
&\le K|\lambda|c_0^2\sqrt{1-s_*^2}.
\end{aligned}
\end{equation}
In particular, 
\begin{equation}\label{eq:ACbound}
\sup_{s_1\le s<1}\|A(s)\|_\oo\le K,\qquad \sup_{s_1\le s<1}\left\|\frac{C(s)}{\la p\ra}\right\|_\oo\le K.
\end{equation}

Since $\D(Z\tga)=0$, the exact weighted equation is
\[
\D(\la p\ra^2Z\tga)=\{\la p\ra^2,\tK\}Z\tga=2\lambda\beta\frac{p(j-p\pr_qE)}{1+p^2}\la p\ra^2Z\tga.
\]
At the observation time, $Z\tga(s_1)=(a\pr_q+c\pr_p)\tga(s_1)$. At the final time, \eqref{eq:AC-A-close}--\eqref{eq:AC-C-close} imply
\begin{equation}\label{eq:tga-first-terminal-error}
\begin{aligned}
&\|\la p\ra^2\bigl[(Z\tga)(1)-(a\pr_q+c\pr_p)\ga_\oo\bigr]\|_r\le K(1+|\lambda|c_0^2)\sqrt{1-s_*^2}\sum_{j+k=1}\|\la p\ra^3\pr_q^j\pr_p^k\ga_\oo\|_r.
\end{aligned}
\end{equation}
Applying \eqref{eq:tga-weighted-transport} with $m=2$ and choosing $(a,c)=(1,0),(0,1)$, we obtain
\begin{equation}\label{eq:tga-first-bound}
\begin{aligned}
&\sum_{j+k=1}\|\la p\ra^2\pr_q^j\pr_p^k\tga(s_1)\|_r\le\sum_{j+k=1}\|\la p\ra^2\pr_q^j\pr_p^k\ga_\oo\|_r+K(1+|\lambda|c_0^2)\sqrt{1-s_*^2}\sum_{j+k=1}\|\la p\ra^3\pr_q^j\pr_p^k\ga_\oo\|_r.
\end{aligned}
\end{equation}

\smallskip
\noindent\textbf{Step 3. Second derivatives and one velocity moment: $l=2$.}
Commuting $Z$ with \eqref{eq:AC-cancellation}, we obtain
\begin{equation}\label{eq:tga-ZA-equation}
\D^2(ZA)=\lambda\alpha\bigl[(\pr_qE)ZA+(\pr_q^2E)A^2\bigr].
\end{equation}
Since $A,C$ are constant in $(q,p)$ at $s=s_1$,
\[
ZA(s_1)=ZC(s_1)=0,\qquad \D(ZA)(s_1)=\lambda\beta(s_1)\pr_q^2E(s_1)a^2.
\]
Two applications of the transport estimate yield
\[
\begin{aligned}
\|ZA(s)\|_\oo \le{}&K|\lambda|c_0^2(s-s_1)\beta(s_1)|a|^2+K|\lambda|c_0^2\int_{s_1}^s(s-u)\alpha(u) \bigl(\|ZA(u)\|_\oo+\|A(u)\|_\oo^2\bigr)\dd u.
\end{aligned}
\]
Using \eqref{eq:AC-kernels}, \eqref{eq:ACbound}, we have
\begin{equation}\label{eq:tga-ZA-bound}
\sup_{s_1\le s<1}\|ZA(s)\|_\oo\le K|\lambda|c_0^2\sqrt{1-s_*^2}.
\end{equation}
Similarly, applying $Z$ to the second equation in \eqref{eq:AC} gives
\[
\begin{aligned}
\D(ZC)=\lambda\beta\bigl[&(\pr_qj-p\pr_q^2E)ZA+(\pr_q^2j-p\pr_q^3E)A^2-2(\pr_q^2E)AC-(\pr_qE)ZC\bigr].
\end{aligned}
\]
Thus
\begin{equation}\label{eq:tga-ZC-equation}
\begin{aligned}
&\D\left(\frac{ZC}{\la p\ra}\right)+\left(\lambda\beta\pr_qE+\D\log\la p\ra\right)\frac{ZC}{\la p\ra}=\lambda\beta\left[\frac{\pr_qj-p\pr_q^2E}{\la p\ra}ZA+\frac{\pr_q^2j-p\pr_q^3E}{\la p\ra}A^2-2(\pr_q^2E)A\frac C{\la p\ra}\right].
\end{aligned}
\end{equation}
The source is bounded by $K|\lambda|c_0^2\beta$, and the initial value is zero. Gr\"onwall gives
\begin{equation}\label{eq:tga-ZC-bound}
\sup_{s_1\le s<1}\left\|\frac{ZC(s)}{\la p\ra}\right\|_\oo\le K|\lambda|c_0^2\sqrt{1-s_*^2}.
\end{equation}
The expansion
\[
Z^2\tga=A^2\pr_q^2\tga+2AC\pr_q\pr_p\tga+C^2\pr_p^2\tga+(ZA)\pr_q\tga+(ZC)\pr_p\tga
\]
gives
\[
\begin{aligned}
(Z^2\widetilde{\gamma})(1)-(a\partial_q+c\partial_p)^2\gamma_\infty&=(A^2-a^2)\partial_q^2\gamma_\infty+2(AC-ac)\partial_q\partial_p\gamma_\infty+(C^2-c^2)\partial_p^2\gamma_\infty \\
&\quad +(ZA)\partial_q\gamma_\infty+(ZC)\partial_p\gamma_\infty .
\end{aligned}
\]
Hence the previous bounds imply
\begin{equation}\label{eq:tga-second-terminal-error}
\begin{aligned}
&\|\la p\ra\bigl[(Z^2\tga)(1)-(a\pr_q+c\pr_p)^2\ga_\oo\bigr]\|_r\le K(1+|\lambda|c_0^2)\sqrt{1-s_*^2}\sum_{1\le j+k\le2}\|\la p\ra^3\pr_q^j\pr_p^k\ga_\oo\|_r.
\end{aligned}
\end{equation}
For example,
\[
\begin{aligned}
&\|A^2-a^2\|_\oo+\left\|\frac{AC-ac}{\la p\ra}\right\|_\oo+\left\|\frac{C^2-c^2}{\la p\ra^2}\right\|_\oo\le K(1+|\lambda|c_0^2)\sqrt{1-s_*^2},
\end{aligned}
\]
which accounts for the coefficients of the second derivatives in this estimate.

At $s=s_1$, the coefficients $A(s_1,q,p)=a$ and $C(s_1,q,p)=c$ are constant in $(q,p)$, so $ZA=ZC=0$. Hence
\begin{equation}\label{eq:Z^2}
    \begin{split}
        (Z^2_{1,0}\tga )(s_1)=\pr_q^2 \tga(s_1),\quad Z^2_{0,1}\tga(s_1)=\pr_p^2\tga(s_1),\quad \frac{1}{2}(Z^2_{1,1}-Z^2_{1,0}-Z^2_{0,1})\tga(s_1)=\pr_q\pr_p\tga(s_1)
    \end{split}
\end{equation}
Applying \eqref{eq:tga-second-terminal-error} to three choices $(a,c)$ therefore gives 
\[
\left\|\langle p\rangle\left[\frac{1}{2}\left[\left(Z_{1,1}^{2}-Z_{1,0}^{2}-Z_{0,1}^{2}\right)\widetilde{\gamma}\right](1)-\partial_q\partial_p\gamma_\infty\right]\right\|_{r}\le K\left(1+|\lambda|c_0^2\right)\sqrt{1-s_*^2}\sum_{1\le j+k\le 2}\left\|\langle p\rangle^3\partial_q^j\partial_p^k\gamma_\infty\right\|_{r}.
\]
Applying the weighted transport estimate \eqref{eq:tga-weighted-transport} with $m=1$ to the following quantities,
\begin{equation*}
    \begin{split}
        Z^2_{1,0}\tga(s), \quad Z^2_{0,1}\tga(s),\quad \frac{1}{2}(Z^2_{1,1}-Z^2_{1,0}-Z^2_{0,1})\tga(s).
    \end{split}
\end{equation*}
 yields
\begin{equation}\label{eq:tga-second-bound}
\begin{aligned}
\sum_{j+k=2}\|\la p\ra\pr_q^j\pr_p^k\tga(s_1)\|_r\le\sum_{j+k=2}\|\la p\ra\pr_q^j\pr_p^k\ga_\oo\|_r+K(1+|\lambda|c_0^2)\sqrt{1-s_*^2}\sum_{1\le j+k\le2}\|\la p\ra^3\pr_q^j\pr_p^k\ga_\oo\|_r.
\end{aligned}
\end{equation}
Summing \eqref{eq:tga-zero-bound}, \eqref{eq:tga-first-bound}, and \eqref{eq:tga-second-bound} over $r\in\{2,\oo\}$ gives
\begin{equation}\label{eq:tga-bootstrap-closure}
\begin{aligned}
&\sum_{l=0}^2\sum_{j+k=l}\sum_{r\in\{2,\oo\}}\|\la p\ra^{3-l}\pr_q^j\pr_p^k\tga(s_1)\|_r\le\left[1+K(1+|\lambda|c_0^2)\sqrt{1-s_*^2}\right]c_0\le2c_0.
\end{aligned}
\end{equation}
The observation time $s=s_1$ is arbitrary, so the same estimate holds throughout the bootstrap interval. It improves \eqref{eq:bootstrap} from $4c_0$ to $2c_0$. Since the final norms are at most $c_0$, the usual continuity argument closes the bootstrap on $[s_*,1]$.
\end{proof}

\subsection{Construction of the local solution}
We now prove Theorem~\ref{thm:3} by Picard iteration. The only new point is that the continuity equation now involves two consecutive iterates.

\begin{proof}[Proof of Theorem \ref{thm:3}]
We divide the proof into four steps.

\smallskip
\noindent\textbf{Step 1. The linear iteration.}
Set $\tga_{(0)}(s,q,p)=\ga_\oo(q,p)$. Given $\tga_{(n)}$, define
\begin{equation}\label{eq:tga-picard-fields}
\rho_{(n)}=\int_\R\tga_{(n)}^2\dd p,\qquad j_{(n)}=\int_\R p\tga_{(n)}^2\dd p,\qquad E_{(n)}=\frac12\iint_{\R^2}\sgn(q-y)\tga_{(n)}(s,y,p)^2\dd y\dd p.
\end{equation}
Construct $\tga_{(n+1)}$ by solving
\begin{equation}\label{eq:tga-picard}
\begin{cases}
\pr_s\tga_{(n+1)}+(p+\lambda\beta E_{(n)})\pr_q\tga_{(n+1)}+\lambda\beta(j_{(n)}-p\rho_{(n)})\pr_p\tga_{(n+1)}=0,\\
\tga_{(n+1)}(1,q,p)=\ga_\oo(q,p).
\end{cases}
\end{equation}
Write the transport operator in this equation as $\D_n$. Since $\pr_qE_{(n)}=\rho_{(n)}$, its field is divergence free. Its coefficients have at most linear growth in $(q,p)$, and on bounded phase-space sets their required derivatives are bounded by a constant times $1+\beta(s)$. Thus the linear final data problem is well defined at $s=1$, since
\[
\int_{s_*}^1(1+\beta(s))\dd s=1-s_*+\sqrt{1-s_*^2}<\oo.
\]
In particular, each iterate has the prescribed final datum. The following estimates are uniform in $n$.

\smallskip
\noindent\textbf{Step 2. Uniform estimates for the iteration.}
Suppose the preceding iterates satisfy \eqref{eq:bootstrap}. Their fields then satisfy \eqref{eq:tga-field-bootstrap}. For $n\ge1$, multiplying the equation for $\tga_{(n)}$ by $2\tga_{(n)}$ and integrating in $p$ gives
\begin{equation}\label{eq:tga-picard-continuity}
\pr_s\rho_{(n)}+\pr_q\bigl(j_{(n)}+\lambda\beta E_{(n-1)}\rho_{(n)}\bigr)=0,\qquad \pr_sE_{(n)}=-j_{(n)}-\lambda\beta E_{(n-1)}\rho_{(n)}.
\end{equation}
The normalized field gives the second identity; the flux is integrable in $q$ because
\[
\|j_{(n)}\|_{L^1_q}\le\|p\tga_{(n)}\|_2\|\tga_{(n)}\|_2, \qquad\|E_{(n-1)}\rho_{(n)}\|_{L^1_q} \le\|E_{(n-1)}\|_\oo\|\tga_{(n)}\|_2^2.
\]
Hence
\[
\D_n\rho_{(n)}=p\pr_q\rho_{(n)}-\pr_qj_{(n)}+\lambda\beta\bigl[(E_{(n)}-E_{(n-1)})\pr_q\rho_{(n)}-\rho_{(n-1)}\rho_{(n)}\bigr].
\]
Construct $A,C$ as in \eqref{eq:AC}, with $\D,E,j$ replaced by $\D_n,E_{(n)},j_{(n)}$. The same commutator calculation gives $[\D_n,Z]=0$, while the second-order equation now reads
\begin{equation}\label{eq:tga-frozen-cancellation}
\D_n^2A=\left[\lambda\alpha\rho_{(n)}+\lambda^2\beta^2\pr_q\bigl((E_{(n)}-E_{(n-1)})\rho_{(n)}\bigr)\right]A, \qquad n\ge1.
\end{equation}
And $\pr_s\rho_{(0)}=0$, and the corresponding identity is
\begin{equation}\label{eq:tga-seed-cancellation}
\D_0^2A=\left[\lambda\alpha\rho_{(0)}+\lambda\beta\pr_qj_{(0)}+\lambda^2\beta^2\pr_q(E_{(0)}\rho_{(0)})\right]A.
\end{equation}
For $n\ge1$, the uniform field bounds imply that, for $k=0,1$,
\[
\left\|\pr_q^k\left[\lambda^2\beta(s)^2\pr_q\bigl((E_{(n)}-E_{(n-1)})\rho_{(n)}\bigr)\right]\right\|_\infty\le K\lambda^2c_0^4\beta(s)^2.
\]
Indeed, using $\pr_qE_{(n)}=\rho_{(n)}$, we obtain
\[
\begin{aligned}
\pr_q\bigl[(E_{(n)}-E_{(n-1)})\rho_{(n)}\bigr]=(\rho_{(n)}-\rho_{(n-1)})\rho_{(n)}+(E_{(n)}-E_{(n-1)})\pr_q\rho_{(n)},
\end{aligned}
\]
and
\[
\begin{aligned}
\pr_q^2\bigl[(E_{(n)}-E_{(n-1)})\rho_{(n)}\bigr]={}&(\pr_q\rho_{(n)}-\pr_q\rho_{(n-1)})\rho_{(n)}+2(\rho_{(n)}-\rho_{(n-1)})\pr_q\rho_{(n)}+(E_{(n)}-E_{(n-1)})\pr_q^2\rho_{(n)},
\end{aligned}
\]
where each product on the right-hand sides is bounded in $L^\infty$ by $Kc_0^4$.

For the first iteration, $n=0$, the additional terms in the coefficient of $A$ are
\[
\lambda\beta(s)\pr_qj_{(0)}+\lambda^2\beta(s)^2\pr_q(E_{(0)}\rho_{(0)}).
\]
The same field bounds give
\[
\begin{aligned}
&\sum_{k=0}^1\left\|\pr_q^k\left[\lambda\beta(s)\pr_qj_{(0)}+\lambda^2\beta(s)^2\pr_q(E_{(0)}\rho_{(0)})\right]\right\|_\infty\le K|\lambda|c_0^2\beta(s)+K\lambda^2c_0^4\beta(s)^2.
\end{aligned}
\]
Their twice-integrated contributions are small, since
\begin{equation}\label{eq:tga-frozen-kernels}
\begin{aligned}
\int_{s_1}^1(1-u)\beta(u)\dd u
&\le(1-s_1)\sqrt{1-s_1^2},\\\int_{s_1}^1(1-u)\beta(u)^2\dd u
&=\int_{s_1}^1\frac{u^2}{1+u}\dd u\le1-s_1\le1-s_*^2.
\end{aligned}
\end{equation}
In particular, the $\beta^2$ contribution is at most
\[
K\lambda^2c_0^4(1-s_*^2)\le K\bigl[(1+|\lambda|c_0^2)\sqrt{1-s_*^2}\bigr]^2.
\]
Thus the proof of \eqref{eq:AC-A-close} applies with the same small bound. Applying $Z$ to \eqref{eq:tga-frozen-cancellation} or \eqref{eq:tga-seed-cancellation} differentiates the bracket only once in $q$, because it is independent of $p$. The estimates for $ZA$ therefore also hold. The equations for $(C-c)/\la p\ra$ and $ZC/\la p\ra$ are unchanged apart from their indices. We obtain
\begin{equation}\label{eq:tga-frozen-coefficients}
\begin{aligned}
\sup_{s_1\le s<1}\left(\|A-a\|_\oo+\left\|\frac{C-c}{\la p\ra}\right\|_\oo+\|ZA\|_\oo+\left\|\frac{ZC}{\la p\ra}\right\|_\oo\right)\le K(1+|\lambda|c_0^2)\sqrt{1-s_*^2}.
\end{aligned}
\end{equation}
Lemma~\ref{lem:tga-bootstrap} now applies to $\tga_{(n+1)}$. Therefore
\begin{equation}\label{eq:tga-picard-uniform}
\begin{aligned}
&\sup_{s_*\le s\le1}\sum_{l=0}^2\sum_{j+k=l}\sum_{r\in\{2,\oo\}}\|\la p\ra^{3-l}\pr_q^j\pr_p^k\tga_{(n+1)}(s)\|_r\\
&\qquad\le\left[1+K(1+|\lambda|c_0^2)\sqrt{1-s_*^2}\right]c_0\le2c_0.
\end{aligned}
\end{equation}
This establishes the uniform estimates on the common interval $[s_*,1]$.

\smallskip
\noindent\textbf{Step 3. Convergence of the iteration.}
For $n\ge1$, write
\[
d_n(s)=\|\tga_{(n)}(s)-\tga_{(n-1)}(s)\|_2+\|\tga_{(n)}(s)-\tga_{(n-1)}(s)\|_\oo.
\]
The field differences satisfy
\[
\|E_{(n)}-E_{(n-1)}\|_\oo\le\frac12\|\tga_{(n)}+\tga_{(n-1)}\|_2\|\tga_{(n)}-\tga_{(n-1)}\|_2\le Kc_0d_n.
\]
Similarly,
\[
\begin{aligned}
\|\rho_{(n)}-\rho_{(n-1)}\|_\oo+\|j_{(n)}-j_{(n-1)}\|_\oo\le Kc_0\|\tga_{(n)}-\tga_{(n-1)}\|_\oo\int_\R(1+|p|)\la p\ra^{-3}\dd p\le Kc_0d_n.
\end{aligned}
\]
Subtracting consecutive equations gives
\begin{equation}\label{eq:tga-picard-difference}
\begin{aligned}
\D_n(\tga_{(n+1)}-\tga_{(n)})=-\lambda\beta\bigl[
(E_{(n)}-E_{(n-1)})\pr_q\tga_{(n)}+\bigl(j_{(n)}-j_{(n-1)}-p(\rho_{(n)}-\rho_{(n-1)})\bigr)\pr_p\tga_{(n)}\bigr].
\end{aligned}
\end{equation}
The final difference is zero. Using Lemma~\ref{lem:Hamilton} and the first-derivative bounds, we obtain
\begin{equation}\label{eq:tga-picard-recurrence}
d_{n+1}(s)\le K|\lambda|c_0^2\int_s^1\beta(u)d_n(u)\dd u,\qquad n\ge1.
\end{equation}
For the first difference, $\D_0(\tga_{(1)}-\ga_\oo)$ has source
\[
-(p+\lambda\beta E_{(0)})\pr_q\ga_\oo-\lambda\beta(j_{(0)}-p\rho_{(0)})\pr_p\ga_\oo.
\]
Its $L^2$ and $L^\oo$ norms are bounded by
$K(c_0+|\lambda|c_0^3)(1+\beta)$, so
\[
d_1(s)\le K(c_0+|\lambda|c_0^3)\sqrt{1-s^2}.
\]
Increasing the  constant $K$, induction in \eqref{eq:tga-picard-recurrence} yields
\begin{equation}\label{eq:tga-picard-factorial}
d_n(s)\le c_0\frac{\bigl[K(1+|\lambda|c_0^2)\sqrt{1-s^2}\bigr]^n}{n!}, \qquad n\ge1.
\end{equation}
Here we used
\[
\int_s^1\beta(u)(1-u^2)^{n/2}\dd u=\frac{(1-s^2)^{(n+1)/2}}{n+1}.
\]
Consequently, for $m>n$,
\[
\begin{aligned}
&\sup_{s_*\le s\le1}\left(\|\tga_{(m)}(s)-\tga_{(n)}(s)\|_2+\|\tga_{(m)}(s)-\tga_{(n)}(s)\|_\oo\right)\le c_0\sum_{k=n+1}^\oo \frac{\bigl[K(1+|\lambda|c_0^2)\sqrt{1-s_*^2}\bigr]^k}{k!} \longrightarrow0.
\end{aligned}
\]
Thus $\tga_{(n)}$ converges to $\tga$ in $C([s_*,1];L^2\cap L^\oo)$, and its fields converge uniformly. The interpolation inequality
\[
\|\nab_{q,p}(\tga_{(n)}-\tga_{(m)})\|_r\le K\|\tga_{(n)}-\tga_{(m)}\|_r^{1/2}\|\nab_{q,p}^2(\tga_{(n)}-\tga_{(m)})\|_r^{1/2},\qquad r\in\{2,\oo\},
\]
and the uniform second-derivative estimates give convergence of the first derivatives as well. Passing to the limit proves \eqref{eq:tK}, the final condition, and the asserted classical regularity for $s<1$. The weighted bounds pass to the limit, with second derivatives understood weakly, and imply \eqref{eq:tga-solution-bound}. The standard regularization of the final datum gives the same conclusions under \eqref{eq:tga-initial}.

\smallskip
\noindent\textbf{Step 4. Endpoint convergence and uniqueness.}
The equation and the uniform bounds imply
\begin{equation}\label{eq:tga-time-bound}
\begin{aligned}
\|\pr_s\tga(s)\|_2+\|\pr_s\tga(s)\|_\oo&\le K(c_0+|\lambda|c_0^3)(1+\beta(s)),\\
\|\tga(s)-\ga_\oo\|_2+\|\tga(s)-\ga_\oo\|_\oo&\le K(c_0+|\lambda|c_0^3)\bigl(1-s+\sqrt{1-s^2}\bigr)\les\sqrt{1-s^2}.
\end{aligned}
\end{equation}
This proves \eqref{eq:tga-endpoint-rate}. The uniform convergence of the first derivatives also gives their continuous endpoint traces.

Finally, let $\tga$ and $\widehat\ga$ be two solutions in the stated class with the same final datum. Repeating the field-difference estimates and \eqref{eq:tga-picard-difference} gives
\[
\begin{aligned}
\|\tga(s)-\widehat\ga(s)\|_2+\|\tga(s)-\widehat\ga(s)\|_\oo\le K|\lambda|c_0^2\int_s^1\beta(u)\left(\|\tga(u)-\widehat\ga(u)\|_2+\|\tga(u)-\widehat\ga(u)\|_\oo\right)\dd u.
\end{aligned}
\]
 Gronwall gives $\tga=\widehat\ga$, proving uniqueness.
\end{proof}

\medskip
Finally, we are ready to prove Theorem \ref{thm:main2}.

\begin{proof}[Proof of Theorem~\ref{thm:main2}]
We first prescribe the final datum for the corrected equation. Set
\begin{equation}\label{eq:wave-terminal-profile}
\ga_\oo(q,p)=\mu_\oo\left(q,p+\frac{\lambda^2}{2}E_\oo(q)\pr_qE_\oo(q)\right).
\end{equation}
The finite momentum translation in \eqref{eq:wave-terminal-profile} is needed to recover the prescribed profile $\mu_\oo$ in the coordinates of \eqref{eq:mscatter}.

\smallskip
\noindent\textbf{Step 1. Application of Theorem~\ref{thm:3}.}
Translation in momentum preserves the spatial density, so
\[
\int_\R\ga_\oo(q,p)^2\dd p=\int_\R\mu_\oo(q,p)^2\dd p,\qquad E[\ga_\oo]=E[\mu_\oo]=E_\oo.
\]
The assumptions on $\mu_\oo$ imply
\[
\|E_\oo\|_\oo+\|\pr_qE_\oo\|_\oo+\|\pr_q^2E_\oo\|_\oo+\|\pr_q^3E_\oo\|_\oo<\oo.
\]
Indeed,
\[
\begin{aligned}
\pr_qE_\oo&=\int_\R\mu_\oo^2\dd p,\\
\pr_q^2E_\oo&=2\int_\R\mu_\oo\pr_q\mu_\oo\dd p,\\
\pr_q^3E_\oo&=2\int_\R\bigl((\pr_q\mu_\oo)^2+\mu_\oo\pr_q^2\mu_\oo\bigr)\dd p.
\end{aligned}
\]
Moreover,
\[
\begin{aligned}
\pr_q(E_\oo\pr_qE_\oo)
&=(\pr_qE_\oo)^2+E_\oo\pr_q^2E_\oo,\\
\pr_q^2(E_\oo\pr_qE_\oo)
&=3(\pr_qE_\oo)(\pr_q^2E_\oo)+E_\oo\pr_q^3E_\oo.
\end{aligned}
\]
Thus the translation in \eqref{eq:wave-terminal-profile} and its first two derivatives are bounded. Since it also preserves the measure $\dd q\dd p$ and
\[
\left\la p+\frac{\lambda^2}{2}E_\oo(q)\pr_qE_\oo(q)\right\ra\simeq\la p\ra,
\]
the chain rule gives
\[
\sum_{j+k\le2}\sum_{r\in\{2,\oo\}}\|\la p\ra^3\pr_q^j\pr_p^k\ga_\oo\|_r<\oo.
\]
We may therefore choose $c_0$ so that \eqref{eq:tga-initial} holds. Theorem \ref{thm:3} yields a solution $\tga$ on $[s_*,1]$ satisfying \eqref{eq:tga-solution-bound} and
\begin{equation}\label{eq:wave-terminal-convergence}
\|\tga(s)-\ga_\oo\|_2+\|\tga(s)-\ga_\oo\|_\oo\les\sqrt{1-s^2}.
\end{equation}

\smallskip
\noindent\textbf{Step 2. Recovery of $\ga$ and continuation to $s=0$.}
Multiplying \eqref{eq:tK} by $2\tga$ and integrating in momentum gives
\[
\pr_s\rho+\pr_q(j+\lambda\beta E\rho)=0.
\]
The flux is integrable in $q$, since
\[
\|j(s)\|_{L^1_q}\le\|p\tga(s)\|_2\|\tga(s)\|_2,\qquad\|E(s)\rho(s)\|_{L^1_q}\le\|E(s)\|_\oo\|\tga(s)\|_2^2.
\]
Consequently, the electric field satisfies
\begin{equation}\label{eq:wave-field-evolution}
\pr_sE=-j-\lambda\beta E\pr_qE.
\end{equation}
Undo the correction by setting
\[
\ga(s,q,p)=\tga\bigl(s,q,p-\lambda\beta(s)E(s,q)\bigr),\qquad s_*\le s<1.
\]
This translation preserves the density and the electric field. Using \eqref{eq:tK}, $\beta'=\alpha$, and \eqref{eq:wave-field-evolution}, we obtain
\[
\begin{aligned}
&(\pr_s+p\pr_q+\lambda\alpha E\pr_p)\ga(s,q,p)\\
&\quad=-\lambda\beta(s)\bigl(\pr_sE+j+\lambda\beta E\pr_qE\bigr)(s,q)\pr_p\tga\bigl(s,q,p-\lambda\beta(s)E(s,q)\bigr)=0.
\end{aligned}
\]
Hence $\ga$ solves \eqref{eq:lensP}. At the fixed time $s_*$, the momentum translation and its spatial derivative are bounded. In particular,
\[
\begin{aligned}
\pr_q\ga(s,q,p)
&=\bigl(\pr_q\tga-\lambda\beta(s)(\pr_qE)\pr_p\tga\bigr)\bigl(s,q,p-\lambda\beta(s)E(s,q)\bigr),\\
\pr_p\ga(s,q,p)&=\pr_p\tga\bigl(s,q,p-\lambda\beta(s)E(s,q)\bigr),
\end{aligned}
\]
and \eqref{eq:tga-solution-bound} implies
\[
\|\la p\ra\ga(s_*)\|_2+\|\la p\ra^3\ga(s_*)\|_\oo+\|\la p\ra\nab_{q,p}\ga(s_*)\|_\oo<\oo.
\]
Proposition \ref{prop:time} applies backwards from $s_*$ to $0$. Indeed, its finite-time estimates use only the boundedness of the time coefficient and the absolute value of $\lambda$, and
\[
\int_0^{s_*}\alpha(u)\dd u=\beta(s_*)<\oo.
\]
We obtain a unique continuation to $s=0$, with
\[
\sup_{0\le s\le s_*}\left(\|\la p\ra\ga(s)\|_2+\|\la p\ra^3\ga(s)\|_\oo+\|\la p\ra\nab_{q,p}\ga(s)\|_\oo\right)<\oo.
\]
Set $\mu_0=\ga(0)$. Then \eqref{eq:data} holds. Since the lens transform is the identity at $s=0$ and maps $[0,\oo)$ onto $[0,1)$, its inverse gives a global classical solution $\mu$ of \eqref{eq:VPH} with initial datum $\mu_0$.

\medskip
\noindent\textbf{Step 3. Expansion of the electric field.}
The normalized field formula and \eqref{eq:wave-terminal-convergence} give
\[
\begin{aligned}
\|E(s)-E_\oo\|_\oo
&\le\frac12\bigl(\|\tga(s)\|_2+\|\ga_\oo\|_2\bigr)\|\tga(s)-\ga_\oo\|_2\les\sqrt{1-s^2},\\
\|\pr_qE(s)-\pr_qE_\oo\|_\oo&\le\|\tga(s)-\ga_\oo\|_\oo\sup_q\int_\R\bigl(|\tga(s,q,p)|+|\ga_\oo(q,p)|\bigr)\dd p\les\sqrt{1-s^2}.
\end{aligned}
\]
Integrating \eqref{eq:wave-field-evolution} from $s$ to $1$ now yields
\begin{equation}\label{eq:wave-field-expansion-final}
\begin{aligned}
E(s,q)&=E_\oo(q)+\int_s^1\bigl(j+\lambda\beta E\pr_qE\bigr)(u,q)\dd u\\
&=E_\oo(q)+\lambda\sqrt{1-s^2}\,E_\oo(q)\pr_qE_\oo(q)+\cO(1-s^2).
\end{aligned}
\end{equation}
The error is uniform in $q$: the current contributes $\cO(1-s)$, while
\[
\begin{aligned}
&\int_s^1\beta(u)\|E(u)\pr_qE(u)-E_\oo\pr_qE_\oo\|_\oo\dd u\les\int_s^1\beta(u)\sqrt{1-u^2}\dd u=\int_s^1u\dd u=\frac{1-s^2}{2}.
\end{aligned}
\]

\medskip
\noindent\textbf{Step 4. Identification of the prescribed scattering profile.}
Use $\Phi$ and $\nu$ from Section~\ref{sec:main}, with the prescribed field $E_\oo$. Thus
\[
\nu(s,q,p)=\tga\left(s,\Phi^1(s,q,p),\Phi^2(s,q,p)-\lambda\beta(s)E\bigl(s,\Phi^1(s,q,p)\bigr)\right).\]
First consider $p=0$. Since
\[
\Phi^1(s,q,0)=q-\lambda(1-s)\beta(s)E_\oo(q),\qquad (1-s)\beta(s)\les\sqrt{1-s^2},
\]
Taylor's formula and \eqref{eq:wave-field-expansion-final} give
\[
\begin{aligned}
E\bigl(s,\Phi^1(s,q,0)\bigr)&=E_\oo(q)+\lambda\bigl[\sqrt{1-s^2}-(1-s)\beta(s)\bigr]E_\oo(q)\pr_qE_\oo(q)+\cO(1-s^2).
\end{aligned}
\]
Using the exact identity
\[
\beta(s)\bigl[\sqrt{1-s^2}-(1-s)\beta(s)\bigr]=\frac{s}{1+s},
\]
we get
\[
\begin{aligned}
&\Phi^2(s,q,0)-\lambda\beta(s)E\bigl(s,\Phi^1(s,q,0)\bigr)=-\lambda^2\frac{s}{1+s}E_\oo(q)\pr_qE_\oo(q)+\cO\bigl(\sqrt{1-s^2}\bigr).
\end{aligned}
\]
For arbitrary $p$, differentiation gives
\[
\pr_p\left[\Phi^2(s,q,p)-\lambda\beta(s)E\bigl(s,\Phi^1(s,q,p)\bigr)\right]=1+\lambda(1-s)\beta(s)\pr_qE\bigl(s,\Phi^1(s,q,p)\bigr).
\]
Integrating this identity from $0$ to $p$ and using the uniform bound for $\pr_qE$, we obtain
\begin{equation}\label{eq:wave-coordinate-errors-final}
\begin{aligned}
|\Phi^1(s,q,p)-q|&\les\sqrt{1-s^2}\la p\ra,\\
\Big|\Phi^2(s,q,p)-\lambda\beta(s)E\bigl(s,\Phi^1(s,q,p)\bigr)&{}-p+\frac{\lambda^2}{2}E_\oo(q)\pr_qE_\oo(q)\Big|\les\sqrt{1-s^2}\la p\ra.
\end{aligned}
\end{equation}
For $s$ sufficiently close to $1$, the momentum brackets along the segment joining these corrected arguments to
\[
\left(q,p-\frac{\lambda^2}{2}E_\oo(q)\pr_qE_\oo(q)\right)
\]
are comparable to $\la p\ra$. Indeed, the finite translation is bounded and the remaining error in momentum is at most $C\sqrt{1-s^2}\la p\ra$. The mean value theorem, the weighted first derivatives of $\ga_\oo$, and \eqref{eq:wave-coordinate-errors-final} therefore give
\[
\begin{aligned}
&\Big|\ga_\oo\left(\Phi^1(s,q,p),\Phi^2(s,q,p)-\lambda\beta(s)E\bigl(s,\Phi^1(s,q,p)\bigr)\right)\\
&\qquad-\ga_\oo\left(q,p-\frac{\lambda^2}{2}E_\oo(q)\pr_qE_\oo(q)\right)\Big|\les\sqrt{1-s^2}\la p\ra^{-2}.
\end{aligned}
\]
By \eqref{eq:wave-terminal-profile}, the second term is exactly $\mu_\oo(q,p)$. Combining this estimate with \eqref{eq:wave-terminal-convergence} proves
\[
\|\nu(s)-\mu_\oo\|_\oo\les\sqrt{1-s^2}.
\]
Finally, the inverse lens transform gives
\[
\nu(\tanh t,q,p)=\mu\bigl(t,\mathcal X(t,q-p,p),\mathcal V(t,q-p,p)\bigr),\qquad \sqrt{1-\tanh^2t}=\frac1{\cosh t}\le2e^{-t}.
\]
This proves \eqref{eq:mscatter} with the prescribed $\mu_\oo$.

\medskip
Finally, for the final datum \eqref{eq:wave-terminal-profile}, uniqueness in Theorem~\ref{thm:3} identifies any two corrected solutions in its stated class on their common endpoint interval. Undoing the momentum correction then identifies their solutions of \eqref{eq:lensP}. Finite-time uniqueness in Proposition~\ref{prop:time} extends this equality to $s=0$. In particular, the initial datum obtained by this construction is independent of the choice of $s_*$, and $\mu_\oo\mapsto\mu_0$ defines a modified wave operator.
\end{proof}

\appendix
\section{Finite-time construction and continuation}\label{app:local}
We construct the Vlasov--Poisson solution and then the Burgers correction by Picard iteration.  

\begin{proposition}\label{prop:time}
Under \eqref{eq:data}, the lens equation \eqref{eq:lensP}, with the symmetrically normalized electric field, has a unique $C^1$ solution on $[0,1)\times\R^2$. For every $T<1$,
\[
 \sup_{0\le s\le T}\left( \|\la p\ra^3\gamma(s)\|_\infty +\|\la p\ra\gamma(s)\|_2+\|\la p\ra\nab_{q,p}\gamma(s)\|_\infty +\|E(s)\|_\infty+\|\pr_qE(s)\|_\infty +\|\pr_q^2E(s)\|_\infty\right)<\infty.
\]
 Moreover, \eqref{eq:Burg} has a unique local solution with two continuous spatial derivatives. It continues past any time $T<1$ at which $\Gamma$ and $\pr_q\Gamma$ remain bounded.
\end{proposition}

\begin{proof}
The proof follows from standard Picard iteration. Fix $T<1$. In this proof, $C_T$ denotes a finite constant depending only on $T$, $C_0$, and $\lambda$, and may increase from line to line.
\smallskip
Set $E_{(0)}(s,q)=E[\mu_0](q)$. 

\noindent\emph{Step 1. Construction of the iterating sequence.}
For $n\ge0$, solve
\begin{equation}\label{eq:app-picard}
\begin{cases}
 \pr_s\ga_{(n+1)}+\{\ga_{(n+1)},\H_n \}=0,\quad \H_n={p^2}/{2}-\lambda\alpha(s)\phi_n(s,q). \\
 \gamma_{(n+1)}(0,q,p)=\mu_0(q,p),
\end{cases}
\end{equation}
Since 
\[
\{\la p\ra^k,\H_n\}=kp{\la p\ra}^{k-2}\lambda\alpha(s)E_n(s,q),
\]
by Lemma \ref{lem:Hamilton}, we have
\[
 \|\la p\ra\gamma_{(n+1)}(s)\|_2 \le\|\la p\ra\mu_0\|_2\exp\left(|\lambda|\int_0^s\alpha(r)\|E_{(n)}(r)\|_\infty\dd r\right).
\]
 For the derivatives, differentiating \eqref{eq:lensH} gives
\begin{equation}\label{eq:linD}
    \begin{split}
        \pr_s (\pr_q\ga_{(n+1)})+\{\pr_q\ga_{(n+1)},\H_n \}&=-\lambda \alpha(s)(\pr_q E_n)\pr_p\ga_{(n+1)},\\
        \pr_s(\pr_p\ga_{n+1})+\{\pr_p\ga_{(n+1)},\H_{n} \}&=-\pr_q\ga_{(n+1)}.
    \end{split}
\end{equation}
Similarly, we have 
\begin{equation}\label{eq:app-square}
 \pr_s(\gamma_{(n+1)}^2)+\pr_q(p\gamma_{(n+1)}^2) +\pr_p(\lambda\alpha E_{(n)}\gamma_{(n+1)}^2)=0.
\end{equation}
In particular, with $M=\|\mu_0\|_2^2$,
\begin{equation}\label{eq:app-mass}
 \|\gamma_{(n+1)}(s)\|_2^2=M,
 \qquad \|E_{(n+1)}(s)\|_\infty\le M/2.
\end{equation}
Integrating \eqref{eq:app-square} in $p$, and then convolving in $q$ with $\frac12\sgn$, yields
\begin{equation}\label{eq:app-continuity}
\begin{split}
 &\pr_s\int\gamma_{(n+1)}^2\dd p  +\pr_q\int p\gamma_{(n+1)}^2\dd p=0,\\
&\pr_qE_{(n+1)}=\int\gamma_{(n+1)}^2\dd p,
 \qquad \pr_sE_{(n+1)}=-\int p\gamma_{(n+1)}^2\dd p.
\end{split}
\end{equation}
Thus the field also satisfies the integral update
\begin{equation}\label{eq:iter}
 E_{(n+1)}(s,q)=E[\mu_0](q)
 -\int_0^s\int_\R p\gamma_{(n+1)}(r,q,p)^2\dd p\dd r.
\end{equation}

\smallskip

\noindent\emph{Step 2. Uniform estimates.}

For $m\ge0$,
\begin{equation}\label{eq:app-weight}
\pr_s (\la p\ra^m\gamma_{(n+1)})+\{\la p\ra^m\ga_{(n+1)},\H_n\}=m\lambda\alpha E_{(n)}\frac{p}{1+p^2}
       \la p\ra^m\gamma_{(n+1)}.
\end{equation}
Lemma \ref{lem:Hamilton}, followed by the weighted $L^2$ identity, gives
\begin{equation}\label{eq:app-uniform-moments}
\begin{aligned}
 \|\la p\ra^3\gamma_{(n+1)}(s)\|_\infty&\le C_0\exp\!\left(\frac32|\lambda|M\beta(s)\right)\le C_T,\\ \frac{\dd}{\dd s}\|\la p\ra\gamma_{(n+1)}(s)\|_2^2
 &=2\lambda\alpha\iint pE_{(n)}\gamma_{(n+1)}^2\dd q\dd p\\
 &\le |\lambda|M^{3/2}\alpha  \|\la p\ra\gamma_{(n+1)}(s)\|_2,\\ \|\la p\ra\gamma_{(n+1)}(s)\|_2
 &\le C_0+\frac{|\lambda|M^{3/2}}2\beta(s)\le C_T.
\end{aligned}
\end{equation}
As a consequence, \eqref{eq:app-uniform-moments} implies
\begin{equation}\label{eq:app-density-bound}
 \|\pr_qE_{(n)}(s)\|_\infty \le C_0^2e^{3|\lambda|M\beta(T)}     \int_\R\la p\ra^{-6}\dd p\le C_T,\qquad n\ge0.
\end{equation}
Multiplying \eqref{eq:linD} by $\la p\ra$ and applying Lemma \ref{lem:Hamilton} to the two equations gives
\begin{equation}\label{eq:app-derivative-bound}
\begin{split}
 &\|\la p\ra\pr_q\gamma_{(n+1)}(s)\|_\infty   +\|\la p\ra\pr_p\gamma_{(n+1)}(s)\|_\infty\\
 &\quad\le 2C_0\exp\!\left(\int_0^s     \bigl[1+|\lambda|\alpha(r)    (\|E_{(n)}(r)\|_\infty+\|\pr_qE_{(n)}(r)\|_\infty)\bigr]\dd r\right) \le C_T.
\end{split}
\end{equation}
Consequently,
\begin{equation}\label{eq:app-field-second}
\pr_q^2E_{(n+1)}=2\int\gamma_{(n+1)}\pr_q\gamma_{(n+1)}\dd p, \qquad \|\pr_q^2E_{(n+1)}\|_\infty \le C_T\int\la p\ra^{-4}\dd p\le C_T.
\end{equation}

\smallskip
\noindent\emph{Step 3. Convergence of the Picard sequence.}
For $n\ge1$, subtraction of two amplitude equations gives
\begin{equation}\label{eq:app-difference}
\begin{split}
 &(\pr_s+p\pr_q+\lambda\alpha E_{(n)}\pr_p)       (\gamma_{(n+1)}-\gamma_{(n)})\\
 &\qquad=-\lambda\alpha(E_{(n)}-E_{(n-1)})\pr_p\gamma_{(n)},\qquad (\gamma_{(n+1)}-\gamma_{(n)})(0)=0.
\end{split}
\end{equation}
Write $e_n(s)=\|E_{(n)}(s)-E_{(n-1)}(s)\|_\infty$. Lemma \ref{lem:Hamilton} and \eqref{eq:iter} imply
\begin{align}
 \|\gamma_{(n+1)}(s)-\gamma_{(n)}(s)\|_\infty
 &\le C_T\int_0^s e_n(u)\dd u,\label{eq:app-gamma-difference}\\
 e_{n+1}(s)
 &\le C_T\int_0^s\|\gamma_{(n+1)}(r)-\gamma_{(n)}(r)\|_\infty\dd r\notag\\
 &\le C_T\int_0^s(s-u)e_n(u)\dd u.\label{eq:app-field-difference}
\end{align}
For the second estimate  we used 
\[
 \int |p|\,|\gamma_{(n+1)}^2-\gamma_{(n)}^2|\dd p \le C_T\|\gamma_{(n+1)}-\gamma_{(n)}\|_\infty                \int |p|\la p\ra^{-3}\dd p.
\]
Since $e_1(s)\le C_Ts$, induction in \eqref{eq:app-field-difference} yields, after fixing one sufficiently large $C_T$,
\begin{equation}\label{eq:app-factorial}
 {\displaystyle
 e_n(s)\le C_T^n\frac{s^{2n-1}}{(2n-1)!},\qquad n\ge1.}
\end{equation}
Indeed,
\[
 \int_0^s(s-u)\frac{u^{2n-1}}{(2n-1)!}\dd u       =\frac{s^{2n+1}}{(2n+1)!}.
\]
Thus both sequences converge uniformly on the whole interval $[0,T]$:
\[
 E_{(n)}\longrightarrow E, \qquad \gamma_{(n)}\longrightarrow\gamma.
\]
Moreover,
\begin{equation}\label{eq:app-field-C1}
 \|\pr_qE_{(n)}-\pr_qE_{(m)}\|_\infty \le C_T\|\gamma_{(n)}-\gamma_{(m)}\|_\infty, \qquad n,m\ge1.
\end{equation}
Hence $\pr_q E$ is continuous, $\pr_q E=\int\gamma^2\dd p$, and the limit satisfies the lens equation distributionally. The uniform first derivative bounds imply that $\gamma$ is Lipschitz in $(q,p)$. The moment bounds pass by Fatou's lemma. Finally, the regularity follows from the method of characteristics.

\smallskip
\noindent\emph{Step 4. Uniqueness and finite-time continuation.}
For two solutions with the same initial data, the preceding difference argument gives
\[
 \|\gamma_1(s)-\gamma_2(s)\|_\infty\le C_T\int_0^s\|E_1(r)-E_2(r)\|_\infty\dd r,
\]
\[
 \|E_1(s)-E_2(s)\|_\infty \le C_T\int_0^s\|\gamma_1(r)-\gamma_2(r)\|_\infty\dd r.
\]
Adding and applying Gronwall gives $\gamma_1=\gamma_2$ and $E_1=E_2$. Because the Picard construction works for every $T<1$, uniqueness makes these solutions agree on overlapping intervals. This constructs $\gamma$ on $[0,1)$ and proves all the asserted finite-time bounds.

\smallskip
\noindent\emph{Step 5. Picard construction of the Burgers correction.}
The field is now fixed. Choose
\[
 R=2\left(1+|\lambda|\sup_{0\le s\le T}\alpha(s)                    \sum_{j=0}^2\|\pr_q^jE(s)\|_\infty\right), \qquad 0<\tau\le\min\{T,(16R)^{-1}\}.
\]
Set $\Gamma_{(0)}=0$ and solve the linear problems
\begin{equation}\label{eq:app-burgers-picard}
(\pr_s+\Gamma_{(n)}\pr_q)\Gamma_{(n+1)}  =\lambda\alpha E, \qquad \Gamma_{(n+1)}(0)=0.
\end{equation}
Their first and second derivative equations are
\begin{equation}\label{eq:app-burgers-derivatives}
\begin{aligned}
(\pr_s+\Gamma_{(n)}\pr_q)\pr_q\Gamma_{(n+1)}+(\pr_q\Gamma_{(n)})\pr_q\Gamma_{(n+1)}
 &=\lambda\alpha \pr_q E,\\
(\pr_s+\Gamma_{(n)}\pr_q)\pr_q^2\Gamma_{(n+1)}+2(\pr_q\Gamma_{(n)})\pr_q^2\Gamma_{(n+1)}&=\lambda\alpha \pr_q^2 E-(\pr_q^2\Gamma_{(n)})\pr_q\Gamma_{(n+1)}.
\end{aligned}
\end{equation}
Suppose that $\sup_{s\le\tau}\max_{0\le j\le2}\|\pr_q^j\Gamma_{(n)}(s)\|_\infty\le R$. Lemma \ref{lem:prop} yields
\begin{align*}
 \sup_{s\le\tau}\|\Gamma_{(n+1)}(s)\|_\infty&\le\tau R/2,\\
 \sup_{s\le\tau}\|\pr_q\Gamma_{(n+1)}(s)\|_\infty&\le(\tau R/2)e^{\tau R},\\
\sup_{s\le\tau}\|\pr_q^2\Gamma_{(n+1)}(s)\|_\infty &\le\tau(R/2+R^2)e^{2\tau R}\le R.
\end{align*}
The first two right-hand sides are also at most $R$, so induction closes.

For the differences, write $\delta\Gamma_{(n)}=\Gamma_{(n)}-\Gamma_{(n-1)}$. Subtraction and one differentiation give
\begin{align*}
(\pr_s+\Gamma_{(n)}\pr_q)\delta\Gamma_{(n+1)}&=-\delta\Gamma_{(n)}\pr_q\Gamma_{(n)},\\
(\pr_s+\Gamma_{(n)}\pr_q)\pr_q\delta\Gamma_{(n+1)}+(\pr_q\Gamma_{(n)})\pr_q\delta\Gamma_{(n+1)}& =-(\pr_q\delta\Gamma_{(n)})\pr_q\Gamma_{(n)} -\delta\Gamma_{(n)}\pr_q^2\Gamma_{(n)}.
\end{align*}
Consequently, with the $W^{1,\infty}$ norm defined as the sum of the two suprema,
\begin{equation}\label{eq:app-burgers-contraction}
\begin{split}
 \sup_{s\le\tau}\|\delta\Gamma_{(n+1)}(s)\|_{W^{1,\infty}}&\le\tau R(1+e^{\tau R})\sup_{s\le\tau}\|\delta\Gamma_{(n)}(s)\|_{W^{1,\infty}}\\
 &\le\frac12\sup_{s\le\tau}                        \|\delta\Gamma_{(n)}(s)\|_{W^{1,\infty}}.
\end{split}
\end{equation}
Thus $\Gamma_{(n)}\to\Gamma$ in $C([0,\tau];W^{1,\infty})$, and $\Gamma$ solves \eqref{eq:Burg}. The second derivative satisfies
\begin{equation}\label{eq:app-burgers-second}
 (\pr_s+\Gamma\pr_q)\pr_q^2\Gamma     +3(\pr_q\Gamma)\pr_q^2\Gamma=\lambda\alpha \pr_q^2 E,\qquad \pr_q^2\Gamma(0)=0.
\end{equation}
The coefficients  are continuous. Lemma~\ref{lem:prop} and bounded weak uniqueness therefore identify $\pr_q^2\Gamma$ with the continuous solution of \eqref{eq:app-burgers-second}. Hence $\Gamma$ has two continuous spatial derivatives. 

Two Burgers solutions satisfy
\[
 (\pr_s+\Gamma_1\pr_q)(\Gamma_1-\Gamma_2)  =-(\Gamma_1-\Gamma_2)\pr_q\Gamma_2.
\]
Gr\"onwall's inequality proves uniqueness. To continue, suppose that $\Gamma$ and $\pr_q \Gamma$ remain bounded as $s\rightarrow T<1$. Equation \eqref{eq:app-burgers-second} gives
\begin{equation}\label{eq:app-burgers-continuation}
 \|\pr_q^2\Gamma(s)\|_\infty\le |\lambda|\int_0^s e^{3\int_r^s\|\pr_q \Gamma(u)\|_\infty\dd u}              \alpha(r)\|\pr_q^2E(r)\|_\infty\dd r<\infty.
\end{equation}
Moreover,
\[
 \pr_s\Gamma=\lambda\alpha E-\Gamma\pr_q\Gamma, \qquad \pr_s\pr_q\Gamma=\lambda\alpha \pr_q E-\Gamma\pr_{q}^2\Gamma-(\pr_q\Gamma)^2.
\]
Thus $\Gamma$ and $\pr_q \Gamma$ have uniform limits at $T$. The continuous-coefficient linear equation \eqref{eq:app-burgers-second}, with these extended coefficients, gives a continuous final value for $\pr_q^2 \Gamma$. Restart \eqref{eq:app-burgers-picard} at $T$ with this $C^2$ datum, including its norm in the choice of $R$ and taking a smaller $\tau$ if needed. This proves the continuation criterion.

Finally, the Hamilton--Jacobi phase is recovered by
\[
 \phi(s,q)=\int_0^qE(s,y)\dd y,\qquad \Theta(s,q)=\int_0^s\left[\lambda\alpha(r)\phi(r,q)-\tfrac12\Gamma(r,q)^2\right]\dd r.
\]
Indeed, the Burgers equation gives
\[
 \pr_q\Theta=\int_0^s(\lambda\alpha E-\Gamma\pr_q\Gamma)\dd r=\Gamma, \qquad \pr_s\Theta+\tfrac12(\pr_q\Theta)^2=\lambda\alpha\phi.
\]
\end{proof}

\end{document}